\documentclass[11pt,a4paper]{article}
\usepackage[USenglish]{babel}
\usepackage{booktabs}
\usepackage[T1]{fontenc}
\usepackage[utf8]{inputenc}
\usepackage{lmodern}
\usepackage{fullpage}
\usepackage{microtype}
\usepackage{csquotes}
\usepackage{authblk}
\usepackage{xcolor}
\usepackage{graphicx}
\usepackage{hyperref}
\usepackage{float}
\usepackage{array}
\usepackage{subfigure}
\usepackage[mathscr]{euscript}
\usepackage{enumitem}
\usepackage{amsmath}
\usepackage{mathtools}
\usepackage{amsthm}
\usepackage{thmtools}
\usepackage{thm-restate}
\usepackage{dsfont}
\usepackage{amssymb}
\usepackage[noabbrev]{cleveref}
\usepackage{crossreftools}
\usepackage{framed}
\usepackage{soul}
\usepackage{cite}
\declaretheorem[style=plain,parent=section,title=Theorem,refname={Theorem,Theorems}]{theo}
\declaretheorem[style=plain,sibling=theo,title=Proposition,refname={Proposition,Propositions}]{prop}
\declaretheorem[style=plain,sibling=theo,title=Corollary,refname={Corollary,Corollaries}]{cor}
\declaretheorem[style=plain,sibling=theo,title=Lemma,refname={Lemma,Lemmas}]{lem}
\declaretheorem[style=definition,sibling=theo,title=Definition,refname={Definition,Definitions}]{defin}
\declaretheorem[style=plain,parent=section,title=Remark,refname={Remark,Remarks}]{rem}

\newlist{myenum}{enumerate}{1}
\setlist[myenum]{label={\bf\roman*)},ref={\bf(\roman*)}}
\newlist{mylist}{itemize}{1}
\setlist[mylist]{label=\textbullet}
\newlist{Hassum}{enumerate}{1}
\setlist[Hassum]{label={\bf(H\arabic*)}}
\crefname{Hassumi}{condition}{conditions}

\crefrangelabelformat{Hassumi}{#3#1--#2#4}
\newcommand{\onlyref}[1]{\labelcref{#1}}

\DeclareMathOperator*{\amin}{arg\,min}

\newcommand{\email}[2][]{\textsuperscript{#1}\href{mailto:#2}{\texttt{#2}}}

\DeclareRobustCommand{\rchi}{{\mathpalette\irchi\relax}}
\newcommand{\irchi}[2]{\raisebox{\depth}{$#1\chi$}}

\newcommand{\RR}{\mathbb{R}}
\newcommand{\NN}{\mathbb{N}}
\newcommand{\Ix}{\mathcal{I}_{\Delta x}}
\newcommand{\Ixt}{\mathcal{I}_{\Delta}}
\newcommand{\Ixtn}{\mathcal{I}_{\Delta_n}}

\newcommand{\It}{\mathcal{I}_{\Delta t}}
\newcommand{\Ixn}{\mathcal{I}_{\Delta x_n}}
\newcommand{\Jxn}{\mathcal{J}_{\Delta x_n}}

\newcommand{\Itn}{\mathcal{I}_{\Delta t_n}}
\newcommand{\Itns}{\mathcal{I}^*_{\Delta t_n}}
\newcommand{\CC}{\mathcal{C}}
\newcommand{\BB}{\mathcal{B}}
\newcommand{\GG}{\mathcal{G}}

\newcommand{\diver}{{\rm{div}}}
\newcommand{\ov}[1]{\overline{#1}}

\def\supp{\mathop{\rm supp}}
\def\P{\mathcal{P}}
\def\M{\mathcal{M}}
\def\S{\mathcal{S}}

\def\dd{{\rm d}}

\title{A semi-Lagrangian scheme for First-Order  Mean Field Games based on monotone operators}
\author[1]{Elisabetta Carlini}
\author[2]{Valentina Coscetti}
\affil[1]{Sapienza, University of Rome, Italy. \textit{Email address:} \email{carlini@mat.uniroma1.it}}
\affil[2]{Sapienza, University of Rome, Italy. \textit{Email address:} \email{valentina.coscetti@uniroma1.it}}
\date{}

\begin{document}

    \maketitle

    \begin{abstract}
       We construct a semi-Lagrangian scheme for first-order, time-dependent, and non-local Mean Field Games.       
       The convergence of the scheme to a weak solution of the system is analyzed by exploiting a key monotonicity property. 
      To solve the resulting discrete problem, we implement a Learning Value Algorithm, prove its convergence, and propose an acceleration strategy based on a Policy iteration method. Finally, we present numerical experiments that validate the effectiveness of the proposed schemes and show that the accelerated version significantly improves performance.
    \end{abstract}

    \medskip
 \noindent{\textbf{AMS-Subject Classification: 91A16, 49N80, 35Q89, 65M12, 65M25.}} 

\section{Introduction} 
Mean Field Games (MFGs) constitute a significant class of systems used to model the strategic interactions of a large number of indistinguishable rational agents. These systems have independently been introduced  by  J.-M. Lasry-Lions 
 \cite{LasryLions06i, LasryLions06ii, LasryLions07} and  by M. Huang, R.P. Malham\'e, and P.E. Caines \cite{HMC06}, and
are typically characterized by a coupling between a backward Hamilton-Jacobi-Bellman (HJB) equation and a forward Fokker-Planck-Kolmogorov-type equation, often referred to as the continuity equation. The state of the system is described by the value function $v(x, t)$, which represents the optimal cost-to-go for an individual agent starting at position $x$ at time $t$, and the population distribution $m(x, t)$, representing the distribution of agents at position $x$ at time $t$.

We are interested in first-order, evolutive, non-local MFGs system of the form:
\begin{align}\label{eq:MFGs1} 
&-\partial_{t}v+H(x,D_x v)=F(x,m(t)) &\textrm{in }&\RR^d\times(0,T),\\
\label{eq:MFGs2} &\partial_{t}m -\diver\Bigl(D_{p}H(x,D_x  v)m\Bigl)= 0&\textrm{in }&\RR^d\times(0,T),\\
\label{eq:MFGs3} &m(0)= m_0^*,\quad v(x,T)= G(x,m(T))&\textrm{in }&\RR^d,
\end{align}
where $T>0$ is a fixed time horizon, $H\colon\RR^d \times\RR^{d}\to\RR$ is convex with respect to its second argument, $F\colon\RR^d\times\P_{1}(\RR^d)\to\RR$,  $G\colon\RR^d\times\P_{1}(\RR^d)\to\RR$, $m_{0}^{*}\in\P_{1}(\RR^d)$, and $\P_{1}(\RR^d)$ is the set of probability measures over $\RR^d$ with finite first order moment.

In general, this complex backward-forward system does not admit an explicit solution, therefore numerical methods are required to compute an approximated one.

For non-degenerate second-order MFGs systems, a variety of numerical techniques have been developed and analyzed:
finite-difference schemes are well-established and investigated in numerous studies \cite{AchdouCapuzzo10,Gueant12,AchdouPorretta16,AlmullaFerreiraGomes,BonnansPfeiffer23,LiuPfeiffer25},  semi-Lagrangian  schemes have also been explored in this context \cite{CarliniSilva18} and \cite{Chowdhury_Jakobsen23} for MFGs with non-local
and fractional diffusion terms,  
deep and reinforcement learning methods have been proposed and studied \cite{Carmona21,Carmona22,Lauriere22},
and recently finite element-based schemes have been introduced for MFGs with non-differentiable Hamiltonians \cite{OsborneSmears24,OsborneSmears25,osbornesmears2025ratesconvergencefiniteelement}.
In cases where the differential game dynamics are deterministic, the resulting MFGs system is of first order. Several numerical methods have been proposed  for approximating the solutions of such systems. Among these are semi-Lagrangian discretizations as discussed in contributions like \cite{camsilva12, CarliniSilva2014,CS15,CarliniSilvaZorkot,AshrafyanGomes}.
Another effective approach involves approximation through discrete-time finite state space MFGs, which has led to significant theoretical and computational developments \cite{HadikhanlooSilva, GianattiSilvaZorkot2024, GianattiSilva2024}. Additionally, Fourier analysis techniques offer an alternative numerical strategy \cite{NurbekyanSaude18, Liu21}.

Our purpose is to construct a semi-Lagrangian (SL) scheme aimed at improving the approach in \cite{CarliniSilva2014,CarliniSilvaZorkot} in two main aspects.
First, we eliminate the convolution-based regularization and the associated $\epsilon$ parameter, which was previously introduced to guarantee the convergence of the gradient of the value function. Second, we provide a convergent algorithm to solve the fully discrete non-linear scheme arising from the numerical approximation.

To define the scheme, we adopt the same idea as in \cite{AshrafyanGomes} for the one-dimensional price model. Although the semi-Lagrangian scheme for the value function remains the same as the one proposed in \cite{CarliniSilva2014, CarliniSilvaZorkot}, the distribution of players is transported by discrete relaxed optimal controls, rather than the regularized gradient of the value function.

The motivation for introducing discrete relaxed optimal controls is due to the fact that we can not ensure uniqueness of the discrete controls. One possible approach is therefore to formulate the scheme using relaxed controls.  Alternatively, one could regularize the Hamiltonian, as in \cite{OsborneSmears25reg}.
 However, this second approach introduces an additional parameter, in a way similar to \cite{CarliniSilva18}.

We prove an existence result  for the proposed scheme, and under a key monotonicity property, its convergence.

It is important to emphasize that the convergence result we obtain refers to a definition of weak solutions based on monotone operators, as introduced in \cite{FerreiraGomes}. This notion is less used compared to the more standard distributional-viscosity solutions, for which the theory is more developed (see, e.g. \cite{cardaliaguet2010notes,LasryLions06i,LasryLions06ii,LasryLions07}). In turn, it allows us to prove a convergence result.

Moreover, a uniqueness result for the discrete system can be established under a more restrictive assumption of the uniqueness of the discrete optimal controls.

The proposed scheme preserves the non-linear backward-forward structure and requires an algorithm to be solved.
We follow the Learning Value Iteration procedure, based on a fictitious play procedure introduced in \cite{CardaliaguetHadikhanloo} in the continuous setting, and we apply this procedure to the proposed  scheme, which results in algorithm referred as \textbf{DLVI}. The convergence of \textbf{DLVI} follows from the results in \cite{HadikhanlooSilva}.

Furthermore, we develop an accelerated version of \textbf{DLVI}, referred as \textbf{ADLVI}, obtained by initializing the \textbf{DLVI} by a efficient initial guess. This guess is computed by a discrete version of the Policy Iteration Algorithm introduced in  \cite{CacaceCamilliGoffi} to solve a second order MFGs problem.

The paper is organized as follows: 
\cref{sec_prel} introduces preliminary results and assumptions, \cref{sec:semi-Lagrangian schemes}
presents the proposed scheme for the MFGs system
and the convergence result, 
\cref{sec:algo} describes the algorithms \textbf{DLVI} and \textbf{ADLVI} and  proves convergence,   finally
\cref{sec:tests} shows the performance of the 
methods in three model problems. 

 \paragraph{Acknowledgments.} 
The authors were partially supported by Italian Ministry of Instruction, University and Research (MIUR) (PRIN Project2022238YY5, ``Optimal control problems: analysis, approximation'') and by INdAM–GNCS Project, codice CUP$\_$E53C24001950001. 
 The second author was partially supported   by European Union - Next Generation EU, Missione 4, Componente 1, CUP$\_$  B53C23002010006 and by Ricerca di Ateneo  CUP$\_$ B83C25004300005.
 
\section{Preliminaries and assumptions}\label{sec_prel}
Let $d\in\NN$. For any $x,y \in \RR^d$, we denote by $x\cdot y$ their standard inner product and by $|\cdot|$ its induced norm. We write $B(0,R)$ and $\ov{B}(0,R)$  for the associated open and closed balls, respectively, centered at the origin and radius $R>0$. We denote by  $|\cdot|_{\infty}$ the maximum norm in $\RR^d$, and we write $B_{\infty}(0,R)$ and $\ov{B}_{\infty}(0,R)$ for the corresponding open and closed balls, respectively, centered at the origin and radius $R>0$. Let $\P(\RR^d)$ be the set of probability
measures on $\RR^d$, endowed with the Kantorovich–Rubinstein distance 
\begin{equation*} \label{e:d_1_alternative} 
d_1(\eta,\tilde \eta)=\sup\bigg\{\int_{\RR^d}\varphi(x)\dd[\eta-\tilde \eta](x)\,\Big|\,\varphi: \RR^d \to \RR  \text{ is 1- Lipschitz } \bigg\}.
\end{equation*} 
We denote by $\to$ and $\rightharpoonup$  the convergence  in $\P(\RR^d)$ with respect to the $d_1$-distance and the narrow topology, respectively.
For every $\eta\in\P(\RR^{d})$ we denote by $\supp(\eta)$ its support.
In what follows, given an absolutely continuous measure (w.r.t. the Lebesgue measure in $\RR^{d}$)  $\eta\in \P(\RR^d)$,  we still denote by $\eta$ its density.
For any normed vector spaces $(E^1,d_{E^1})$ and $(E^2,d_{E^2})$, we denote by  $\mathcal{B}(E^1;E^2)$, $\CC_b(E^1;E^2)$, $\CC^k(E^1;E^2)$, $\CC^{k}_c(E^1;E^2)$, 
the sets of  bounded, bounded continuous, $k$-times continuously differentiable, and $k$-times continuously differentiable functions with compact support   from $E^1$ to $E^2$ respectively.  Whenever $E^2=\RR$, we suppress explicit mention of the codomain and use the abbreviated notations  $\mathcal{B}(E^1)$, $\CC_b(E^1)$,  $\CC^k(E^1)$ and   $\CC^{k}_c(E^1)$. 

We suppose that the functions $F,G :\RR^{d} \times \P_{1}(\RR^d) \to \RR$ and the measure $m^*_0$, which are the data of   the MFGs system, satisfy the following assumptions:
   \begin{Hassum}
        \item\label{continuity} $F,G \in \CC( \RR^d \times\P_{1}(\RR^d))$; 
        \item\label{C2bounded}  $G(\cdot,\eta)\in \CC^2(\RR^d)$ for any $\eta \in \P_1(\RR^d)$, and there exist  some constants $C_{F,i}>0,C_{G,i}>0$ $(i=1,2)$ such that, for any $\eta \in \P_1(\RR^d)$, 
        \begin{equation*}
        \begin{aligned}
        |F(x,\eta|&\leq C_{F,1}, \\ 
            |F(x,\eta)-F(y,\eta|&\leq C_{F,2}|x-y|, \\ 
        \supp(G(\cdot,\eta))&\subseteq \ov B(0,C_{G,1}),\\ 
            \sup_{x\in \RR^{d}}\Big\{|G(x, \eta)|+ |D_x G(x, \eta)|+ |D^2_xG(x, \eta)|\Big\}\,&\leq C_{G,2};\\
        \end{aligned}
        \end{equation*}
     \item\label{initialdensity} the initial condition $m^*_{0}\in \P_{1}(\RR^d)$ is absolutely continuous with respect to the Lebesgue measure, with density still denoted by $m^*_{0}$ such that $m_0^*\in \CC^2(\RR^d)$   and $\supp(m^*_{0})\subseteq \ov B(0,C^*)$, for some $C^*>0$;  
     \item \label{ass:lagrangian}  it holds that 
\begin{equation*}
    H(x,p):=\sup_{a\in \RR^d}\big( a\cdot p - L(x,a)\big) \quad \text{ for any\,} (x,p) \in \RR^d\times\RR^d,
\end{equation*}
     where $L\in \CC^2(\RR^d\times\RR^d)$, it is bounded from below  and there exist some constants $C_{L,i}>0$  $(i=1,\dots,5)$ such that, for any $(x,a)\in \RR^d \times \RR^d$, we have
\begin{equation*}
\begin{aligned}
L(x,a)&\leq C_{L,1}|a|^2+C_{L,2},\\
|D_{x}L(x,a)|&\leq C_{L,3}(1+|a|^2),\\
C_{L,4}|b|^2&\leq 
D^{2}_{aa}L(x,a)b \cdot b \quad\text{for any }b\in\RR^d,\\
D^{2}_{xx}L(x,a)y\cdot y &\leq C_{L,5}(1+|a|^2)|y|^2\quad\text{for any }y\in\RR^d;
\end{aligned}
\end{equation*}
\item\label{monoticityassumption} for $h=F,G$ the following Lasry--Lions monotonicity condition holds         \begin{equation*}
    \int_{\RR^{d}}\Big(h(x,\eta)-h(x,\tilde \eta)\Big)\dd[\eta-\tilde \eta](x)\geq 0 \quad \text{for any }  \eta,\tilde \eta\in \P_{1}(\RR^d).
\end{equation*}
\end{Hassum}
\begin{rem} As shown in \cite[Remark 2.1]{CarliniSilvaZorkot}, \ref{ass:lagrangian} implies that there exist $C_{L,6}>0 $ and $C_{L,7}>0$ such that 
\begin{equation*}
    L(x,a)\geq C_{L,6}|a|^{2}-C_{L,7}\quad\text{for any }x,\,a\in\RR^d.
\end{equation*}
Therefore, $H\in \CC^2(\RR^d\times \RR^d)$, $D_{p}H(x,p)$ is the unique maximizer of $\sup_{a\in\RR^{d}}\big\{ a\cdot p -L(x,a)\big\}$ and there exist some constants $C_{H,i}>0$ $(i=1,\cdots,5)$ such that 
\begin{equation*}
\begin{aligned}
   C_{H,1}|p|^2-C_{H,2}&\leq H(x,p)\leq C_{H,3}|p|^2+C_{H,4}\quad\text{for any }x,\,p\in\RR^d, \\ 
   |D_{p}H(x,p)|&\leq C_{H,5}(1+|p|)\quad\text{for any }x,\,p\in\RR^{d}.
\end{aligned}
\end{equation*}
\end{rem}

\begin{rem}\label{rem:viscosity} Under
\cref{C2bounded,continuity,initialdensity,ass:lagrangian},there exists at least one solution 
$(m,v) \in L^1(\mathbb{R}^d\times(0,T))\times W^{1,\infty}_{loc} (\mathbb{R}^d\times[0, T])$ such that \eqref{eq:MFGs2}  is satisfied in the sense of distributions, while \eqref{eq:MFGs1} is satisfied
 in the viscosity sense. Moreover, also assuming \ref{monoticityassumption}
and  $F(\cdot,\eta)\in \CC^2(\RR^d)$ with uniformly bounded derivatives, the solution is unique (see \cite{lions2007cours,cardaliaguet2010notes}).
\end{rem}

A notion of weak solution for stationary MFGs problems was introduced in \cite{FerreiraGomes}, relying on the theory of motonone operators, as mentioned also in \cite{CardaliaguetPorretta} for second order MFGs problems. Following this approach,   the MFGs system can be rephrased as $\mathcal{A}[w] = 0$ subject to  initial-terminal conditions, where $\mathcal{A}$ is  the
operator on the couple $w=(m,v)$ defined by
\begin{equation*}
\mathcal{A}[w]
\coloneqq
\begin{bmatrix}
	\partial_{t}v(x,t)-H(x,D_xv(x,t))+F(x,m(t)) \\
	\partial_{t} m(x,t)- \diver\Bigl(D_pH\big(x,D_xv(x,t)\big)m(x,t)\Bigl) \\
\end{bmatrix}. 
\end{equation*}
For any $w=(m,v)$ and $\Tilde w=(\Tilde m,\Tilde v)$, setting $\langle w,\tilde{w} \rangle \coloneqq \int_{Q_T}\Big( m(x,t)\Tilde m(x,t) +  v(x,t)\Tilde v(x,t)  \Big)\dd x\dd t,$
and $Q_T:={\RR^d\times[0,T]}$,
we give the following notion of weak solution:
\begin{defin}\label{weak solution}
    The couple $w^*=(m^*,v^*)$ is a {\emph weak} solution  to   the MFGs system if $m^*,v^* :{Q_T}\to \RR$, $m^*(t)\in\P_1(\RR^d)$ for any $t\in [0,T]$, 
     $ v^*\in L^\infty({Q_T})$ and 
    \begin{equation*}
         \langle \mathcal A[\varphi],\varphi-w^* \rangle \geq 0 \qquad \text{for any } \varphi= (\phi,\psi) \in \mathcal{D},
    \end{equation*}
where 
       \begin{equation*}
               \begin{aligned}
        \mathcal{D}\coloneqq \Big\{ \varphi=(\phi,\psi) \,  \Big |\,  
        &\phi,\psi\in \CC_c^2({Q_T}), \\ 
        &\phi \geq 0, \, \int_{\RR^d}\phi(x,t)\dd x=1  \text{ for any } t \in [0,T], \, \phi(x,0)=m_0^*(x),\\  & \psi(x,T)=G(x,\phi(T)), \, 
       ||D_x \psi||_{L^\infty({Q_T})}\leq \ov  C_{\text{{\rm b}}} \Big\},
     \end{aligned}
    \end{equation*}
with $\ov C_{\text{{\rm b}}}>0$ such that $C_{G,2}\leq\ov C_{\text{{\rm b}}}\leq C_H/C_{H,5}-1$, where $ C_H\coloneqq C_{\text{{\rm b}}}+C_{H,5}(C_{G_2}+1)$ and  $ C_{\text{{\rm b}}}>0$ is the constant defined in  \cite[Proposition 2.1]{CarliniSilvaZorkot}. 
\end{defin}

Let us observe that the constants  $\ov C_{\text{{\rm b}}}$ and $C_H$ guarantee  $||D_pH(\cdot,D_x\psi)||_{L^\infty({Q_T})}\leq C_H$.

The notion of solution in \cref{weak solution}
is less used than the one referred in \cref{rem:viscosity}, however it allowed us to
established a converge result for
 the SL scheme proposed in \cref{sec: A semi-Lagrange scheme for MFGs}.  In addition, this convergence result
 proves that, under   the same assumptions of \cref{rem:viscosity}  and also assuming  \ref{DiscreteMonotonicity} 
and \ref{condition dxdt for compact support discrete solution},  there exists a weak solution to   the MFGs system (\cref{existencewealsolution}) in the sense of \cref{weak solution}.
However, we are currently unable to conclude whether such a weak solution is also a distributional-viscosity solution; this remains an open question. Partial results in this direction have been 
established for stationary MFGs in \cite{FerreiraGomesVardan25}. 
We also note that any classical solution is automatically both a weak solution and a distributional-viscosity solution, as also  observed in \cite{AshrafyanGomes}.

\section{Semi-Lagrangian schemes}\label{sec:semi-Lagrangian schemes}
We recall some standard SL schemes for approximating solutions to the HJB equation and the continuity equation. These schemes will be employed in the subsequent algorithms developed to compute approximate solutions to   the MFGs system.

Let us consider a pair $\Delta=(\Delta x,\Delta t)>0$ such that $N_T\Delta t=T$ for some $N_T\in\NN_+$. We define the sets of indices  
$\It\coloneqq\{0,\dots, N_T\}$, $\It ^*\coloneqq\It \backslash \{N_T\}$,  
$\Ix\coloneqq\mathbb{Z}^d$, $\Ixt\coloneqq \Ix\times\It$ and $\Ixt^*\coloneqq \Ix\times\It^*$.
For any $(i,k)\in \Ixt$ we define  $x_i\coloneqq i\Delta x$, $t_k\coloneqq k\Delta t$ and the cell $E_{i}:= \{x\in\RR^d\, |\, |x-x_i|_\infty \leq \Delta x/2\}$. We define the $d$ dimensional space grid
$\GG_{\Delta x}\coloneqq \{x_{i},  i\in \Ix\}$, the time grid $\GG_{\Delta t}\coloneqq \{t_{k}, k\in \It\}$ and the space-time grid $\GG_{\Delta}\coloneqq \GG_{\Delta x}\times \GG_{\Delta t}$. We also define  $\GG^*_{\Delta t}\coloneqq \{t_{k}=k\Delta t, k\in \It^*\}$ and $\GG^*_{\Delta}\coloneqq \GG_{\Delta x}\times \GG^*_{\Delta t}$.   Given $h \in \BB(\GG_{\Delta x};\RR^d)$ and $l\in \BB(\GG_{\Delta};\RR^d)$, we use the notation 
\begin{equation*}
    h_i\coloneqq h(x_i) \qquad l_{i,k}\coloneqq l(x_i,t
    _k) \qquad \text{ for any } (i,k)\in \Ixt,
 \end{equation*} 
 and, for $k\in \It$, we define the function $l_{k}\in \BB(\GG_{\Delta x};\RR^d)$ such that  $ (l_{k})_{i}\coloneqq l_{i,k}$ for any $i\in \Ix$.
 Given $\varphi \in \CC_{b}(\RR^{d})$ and $\psi\in \CC_{b}(Q_T)$, we define the functions  $\hat \varphi \in\BB(\GG_{\Delta x})$ and $\hat \psi \in \BB(\GG_{\Delta})$ such that 
\begin{equation*}
  \hat \varphi _i\coloneqq \varphi(x_i) \qquad \hat \psi _{i,k}\coloneqq\psi(x_i,t_k) \qquad \text{ for any } (i,k)\in \Ixt.
\end{equation*}
We consider the $\mathbb{Q}_1$ basis of functions $\{\beta_{i}:\RR^{d}\to\RR, \, i\in\Ix\}$ (see \cite{CarliniSilvaZorkot}) and the interpolation operator $I_{\Delta x}:\BB(\GG_{\Delta x})\to \CC_{b}(\RR^{d})$ defined by  
\begin{equation}\label{interpolation operator}
    I_{\Delta x}[W](\cdot)\coloneqq \sum_{i\in \Ix}W_{i}\beta_{i}(\cdot).
\end{equation}
We have that, given $\phi\in \CC(\RR^d)$, the following holds  (see \cite{QuarteroniValli})
\begin{equation} \label{interperr}
    \sup_{x\in \RR^d}|I_{\Delta x}[\hat{\phi}](x)-\phi(x)|= \mathcal{O}(\Delta x^{\gamma}),
\end{equation}
where $\gamma=1$ if $\phi$ is Lipschitz and $\gamma=2$ if $ \phi\in \CC^2(\RR^d)$ with bounded first and second
derivatives.

\subsection{Fully discrete scheme for the Hamilton Jacobi Bellman equation }\label{A_fully_discrete_scheme_for_the_Hamilton_Jacobi_Bellman_equation}
We consider the following fully discrete SL scheme, which differs from standard SL formulations in the discretization of the running and terminal costs.
Although the change is mild, it enables us to establish a discrete monotonicity property for the entire scheme, a key ingredient in the analysis of the Mean Field Games system.
In particular, \cref{rem: discrete monotonicity}  and \cref{rem: monotone descrete operator} show that the proposed discretization satisfies the required discrete monotonicity assumptions.
Given $u \in \BB(\GG_{\Delta})$, we define its extension on $ Q_T $  by
\begin{equation}\label{vdelta}
    V_{\Delta}[u](x,t)\coloneqq I_{\Delta x}[u_{\lceil\frac{t}{\Delta t}\rceil}](x).
\end{equation}
Given $\eta \in \CC([0,T];\P_1(\RR^d))$ 
and  $C_H$  as in  \cref{weak solution},  we define ${S}_{\Delta}[\eta]:\BB(\GG_{\Delta x})\times \Ixt \to \RR$ by
\begin{equation*}\label{eq:infS}
    {S}_{\Delta}[\eta](W,i,k)\coloneqq \inf_{a\in \ov{B}_{\infty}(0,C_H) }\left\{I_{\Delta x}[W](x_{i}-\Delta t a)+\Delta tL(x_i,a)\right\}+\frac{\Delta t}{\Delta x^d}\int_{E_i} F(x,\eta(t_{k}))\dd x.
\end{equation*}

We consider the following fully-discrete SL scheme for \eqref{eq:MFGs1}: define recursively 
 $v\in \BB(\GG_{\Delta})$ as
\begin{equation}\label{eq:HJBdiscr}\tag{$HJBdiscr$}
            \begin{aligned}
               &v_{i,k}={S}_{\Delta}[\eta](v_{k+1},i,k) &&\text{for any   }(i,k)\in \Ixt ^*, \\
               &v_{i,N_T}=\int_{E_i}G(x,\eta(T))\dd x &&  \text{for any  }i\in \Ix. \\  
            \end{aligned}
\end{equation}
and its extension  $v_{\Delta}:=V_{\Delta}[v]$.
The following properties of ${S}_{\Delta}$ and $v_{\Delta}$  are proved similarly to \cite{CarliniSilvaZorkot}. 
\begin{prop}\label{properties decrete value function}
Under \onlyref{continuity,C2bounded,ass:lagrangian}, let $\eta \in \CC([0,T];\P_1(\RR^d))$ and $\Delta=(\Delta x,\Delta t)>0$,
the following properties hold:
\begin{enumerate}
\item [{\rm(i)}]  {\rm[Monotonicity]} 
For any $W,\tilde W\in \BB(\GG_{\Delta_x})$ such that $W_i\leq \tilde W_i $ for any $i\in \Ix$, we have 
\begin{equation*}
{S}_{\Delta}[\eta](W,i,k)\leq {S}_{\Delta}[\eta](\tilde W,i,k) \quad \text{for any } (i,k)\in \Ixt.
\end{equation*}
\item [{\rm(ii)}]  {\rm[Consistency]}  
For any $\psi\in\CC^2(Q_T)$ with bounded derivatives and s.t. 
$||D_x \psi||_{L^\infty(Q_T)}\leq \ov  C_{\text{{\rm b}}}$, for any $(i,k)\in \Ixt^*$, we have 
\begin{equation*}
\begin{aligned}
\frac{\hat\psi_{i,k}-{S}_{\Delta}[\eta](\hat{\psi}_{k+1},i,k)}{\Delta t}=&\partial_{t}\psi(x_i,t_k)+H(x_i,D_{x}\psi(x_i,t_k))-F(x_i,\eta(t_k))\\ &+\mathcal{O}\left(\frac{\Delta x^2}{\Delta t}+\Delta t +\Delta x \right).
\end{aligned}
\end{equation*}

    \item [{\rm(iii)}] {\rm[Stability]}
  We have 
\begin{equation*}
|v_{\Delta}(x,t)|\leq \widetilde C_{\text{{\rm v}}}\quad\text{for any }(x,t)\in \RR^{d}\times [0,T], 
\end{equation*}
where $\widetilde{C}_{\text{{\rm v}}}>0$ is independent of $(\eta,\Delta x,\Delta t)$.
\item [{\rm(iv)}] {\rm[Lipschitz property]}
  We have 
  \begin{equation*}
      \big|v_{\Delta}(x,t)-v_{\Delta}(y,t)\big|\leq \widetilde{C}_{\text{{\rm Lip}}}|x-y|\quad\text{for any }x,\,y\in\RR^d\,, t\in[0,T],
  \end{equation*}
  where $\widetilde{C}_{\text{{\rm Lip}}}>0$ is independent of
$(\eta,\Delta x,\Delta t)$. 
\end{enumerate}
\end{prop}
\medskip

The  term $\mathcal{O}\Delta x)$ 
in the consistency error is due to the introduction of the cell average of the cost instead of  their pointwise value. However, this extra term does not change the overall truncation error of the fully discrete MFGs system, which appears in the proof \cref{convergenceschemeMFGsD1}.

 For any
 $(i,k)\in \Ixt^*$ and $v \in \BB(\GG_{\Delta})$, we define  the set of {\emph{discrete optimal controls}}
\begin{equation}\label{optimal set}
\begin{aligned}
  \Lambda_{i,k}[v]&\coloneqq\amin_{a \in \ov{B}_{\infty}(0,C_H)}\left\{I_{\Delta x}[v_{k+1}](x_{i}-\Delta t a)+\Delta tL(x_i,a)\right\}.\\ 
\end{aligned}
\end{equation}

Although the Lagrangian is strictly convex and therefore ensures uniqueness of the optimal control at the continuous level, this property may be lost in the discrete scheme. In fact, the interpolation procedure used in \eqref{optimal set} does not preserve strict convexity, and as a consequence the discrete minimization problem may admit multiple solutions. To deal with this potential lack of uniqueness, following \cite{AshrafyanGomes}, we introduce relaxed controls, i.e., probability measures supported on the set of discrete minimizers. This formulation naturally extends deterministic controls and guarantees the well-posedness of the scheme.

Therefore, we also define a set formed by grid functions of optimal controls and a set of  {\emph{relaxed controls}}
\begin{equation*} 
\begin{aligned}
  \Lambda_{\Delta}[v]&\coloneqq\{q\in \BB(\GG_{\Delta}^*;\RR^d) \, | \, q_{i,k}\in \Lambda_{i,k}[v]\}, \\ 
   \M_{\Delta}[v]&\coloneqq \Big\{\mu=(\mu_{i,k})_{(i,k)\in \Ixt^*} \, \Big | \, \mu_{i,k}\in \P_1(\RR^d), \, \supp(\mu_{i,k})\subseteq\Lambda_{i,k}[v]  \Big\}.
\end{aligned}
\end{equation*}

\subsection{Fully discrete scheme for the continuity equation}\label{A_fully_discrete_scheme_for_the_continuity_equation}
Given $\Delta=(\Delta x, \Delta t)$, we define  the sets $ \S(\GG_{\Delta x}):= \{ z\in \BB(\GG_{\Delta x}) |\, z_{i}\geq 0, \; \sum_{i\in \Ix}z_{i}=1 \}$ 
and $\S(\GG_{\Delta})\coloneqq \{ \eta=(\eta_{k})_{k=0}^{N_{T}} |\,\eta_{k}\in \S(\GG_{\Delta x}) \}$.
For any $\eta\in \S(\GG_{\Delta x})$, we define its extension on  $\RR^d$ as 
\begin{equation*}
    M_{\Delta x }[\eta](x)\coloneqq \frac{1}{(\Delta x) ^{d}}\sum_{i \in \Ix} \eta_i\rchi_{E_{i}}(x), \quad {\textrm{for }}\,x\in \RR^d,
\end{equation*}
where
$ E_{i}\coloneqq \{x\in\RR^d\, |\, |x-x_i|_\infty \leq \Delta x/2\}$,  and $\rchi_{E}$ denotes the characteristic function of $E\subset \RR^d$.
Given $\eta\in \S(\GG_{\Delta})$, we also define  its extension on  $Q_T$  as 
\begin{equation}\label{mdelta}
 M_{\Delta}[\eta](x,t)\coloneqq M_{\Delta x}[\eta_k](x) \quad \text{if }t\in [t_k,t_{k+1}), k \in \It^*,\quad M_{\Delta}[\eta](x,t)\coloneqq M_{\Delta x}[\eta_{N_T}](x) \quad \text{if }t=T.
\end{equation}
Given $\mu=(\mu_{i,k}) $ such that $\mu_{i,k}\in \P(\RR^d)$ for  any ${(i,k)\in \mathcal{I}_{\Delta}^*}$, we define recursively $m\in \S(\GG_{\Delta})$  by the following discrete transport of the initial measure $m_0^*$ induced  by a generic probability measure $\mu$ :
\begin{equation}\label{eq:CEdiscr}\tag{$CEdiscr$}
            \begin{aligned}
               &m_{i,k+1}=\sum_{j\in \Ix}m_{j,k} \int_{\RR^d}\beta_{i}(x_j-\Delta t a)\dd \mu_{j,k}(a) &&\text{for any   }(i,k)\in \Ixt ^*, \\
               &m_{i,0}=\int_{E_{i}}\dd [m_0^*](x) &&  \text{for any  }i\in \Ix. \\  
            \end{aligned}
\end{equation}
It is easy to verify the following result,  analogous to \cite[Lemma 4.1]{CarliniSilvaZorkot}. 
\begin{lem}\label{SLCEproperties}
Let us assume  \ref{initialdensity} and
  $\Delta=(\Delta x,\Delta t )>0.$
  Let $\mu=(\mu_{i,k})$ be such that $\mu_{i,k}\in \P(\RR^d)$ 
  with $\supp(\mu_{i,k}) \subseteq \ov B(0,R)$, for some $R>0$, 
  for any ${(i,k)\in \mathcal{I}_{\Delta}^*}$. Let
  $m\in \S(\GG_{\Delta})$  be defined by \eqref{eq:CEdiscr}. 
For any $c>0$, such that $\Delta x\leq c\Delta t$, then there exists $C>0$, independent of $(\Delta ,\mu)$, such that 
  $m_{i,k}=0$ if $x_i\notin\ov B(0,C)$ for any $(i,k) \in \Ixt$.
\end{lem}

\subsection{A semi-Lagrangian scheme \eqref{eq:MFGsdiscr1}  for MFGs}\label{sec: A semi-Lagrange scheme for MFGs}
We introduce a SL scheme to approximate a weak solution to   the MFGs system. 
This scheme is obtained by combining the SL scheme for the HJB equation defined in \cref{A_fully_discrete_scheme_for_the_Hamilton_Jacobi_Bellman_equation} and the SL scheme for the continuity equation defined in \cref{A_fully_discrete_scheme_for_the_continuity_equation}. We show that, under suitable assumptions, there exists a  solution to this scheme and it converges to a weak solution to   the MFGs system. 

We define $f,g : \GG_{\Delta x} \times \S(\GG_{\Delta x})\to\RR$ 
by 
\begin{equation*} 
    f(x_i,\eta)\coloneqq \frac{1}{\Delta x^d}\int_{E_i}F(x,M_{\Delta x}[\eta])\dd x, \qquad g(x_i,\eta)\coloneqq \frac{1}{\Delta x^d}\int_{E_i}G(x,M_{\Delta x}[\eta])\dd x. 
\end{equation*}
The proposed SL scheme is the following: find $(m,v) \in \S(\GG_{\Delta})\times \BB(\GG_{\Delta})$ such that 
\begin{equation}\label{eq:MFGsdiscr1}\tag{$MFG^1_d$}
            \begin{aligned}
               &v_{i,k}=I_{\Delta x}[v_{k+1}](x_{i}-\Delta t q_{i,k})+\Delta tL(x_i,q_{i,k})  +\Delta t f(x_i,m_k)&&\text{for any   }(i,k)\in \Ixt ^*, \\
                 &m_{i,k+1}=\sum_{j\in \Ix}m_{j,k} \int_{\RR^d}\beta_{i}(x_{j}-\Delta t a)\dd[\mu_{j,k}](a)  &&\text{for any   }(i,k)\in \Ixt ^*,\\
               &v_{i,N_T}=g(x_i, m_{N_T}), \qquad m_{i,0}=\int_{E_{i}}\dd [m^*_{0}](x)&&  \text{for any  }i\in \Ix,\\  
            \end{aligned}
\end{equation}
 with  $q \in \Lambda_{\Delta}[v]$  and $\mu\in \M_{\Delta }[v]$.
We consider the following assumptions:
\begin{Hassum}[resume]
\item \label{DiscreteMonotonicity} for $h=f,g$ the following monotonicity condition holds   
  \begin{equation*}
  \sum_{i \in \Ix}\Big(h(x_i,\eta)- h(x_i,\tilde \eta)  \Big)\Big(\eta_i-\tilde \eta_i\Big)\geq 0  \,\text{ for any } \eta,\tilde \eta\in \S(\GG_{\Delta x}) \text{ with compact support};
\end{equation*}
\item \label{condition dxdt for compact support discrete solution}
the pair $\Delta=(\Delta x,\Delta t)$ of spatial and time step is such  that $ \Delta x < c \Delta t$, for some $c>0$.
\end{Hassum}
\begin{rem}\label{rem: discrete monotonicity}  Under  \ref{monoticityassumption}, assumption \ref{DiscreteMonotonicity} is automatically satisfied. 
Indeed, for $\eta,\tilde \eta$ as in \ref{DiscreteMonotonicity}, the definition of 
$f$ yields that 
\begin{equation*}
     \begin{aligned}
       & \sum_{i\in \mathcal{I}_{\Delta x}}\Big( f(x_i,\eta)-f(x_i,\tilde \eta)\Big)(\eta_i-\tilde\eta_i)\\ 
          & = \sum_{i\in \mathcal{I}_{\Delta x}}\int_{E_i}\Big( F(x,M_{\Delta x}[\eta])- F(x,M_{\Delta x}[\tilde \eta])\Big)\Big (M_{\Delta x}[\eta](x)-M_{\Delta x}[\tilde\eta]\Big)\dd x\\ 
             & =\int_{\RR^d}\Big( F(x,M_{\Delta x}[\eta])- F(x,M_{\Delta x}[\tilde \eta])\Big)\Big (M_{\Delta x}[\eta](x)-M_{\Delta x}[\tilde\eta]\Big)\dd x\geq 0,\\
     \end{aligned}
 \end{equation*}
 where the last inequality follows from \ref{monoticityassumption}. A similar argument holds for $h = g$ in \ref{DiscreteMonotonicity}.
\end{rem}

For any   $v \in \BB(\GG_{\Delta})$, $q\in \Lambda_{\Delta }[v]$ and $\mu\in \M_{\Delta}[v]$, we notice that, for any $(i,k)\in \Ixt^*$,
\begin{equation}\label{optimalrelaxedcontrolproperty}
  I_{\Delta x}[v_{k+1}](x_{i}-\Delta t q_{i,k})+\Delta tL(x_i,q_{i,k}) =\int_{\RR^d}\Big(I_{\Delta x}[v_{k+1}](x_{i}-\Delta t a)+\Delta tL(x_i,a) \Big)\dd[\mu_{i,k}](a). 
\end{equation}
We observe that, for any $v,\tilde v \in \BB(\GG_{\Delta})$, setting $\mu\in \M_{\Delta}[v]$ and $\tilde \mu\in \M_{\Delta}[\tilde v]$,
    the following holds  
\begin{equation}\label{minproperty}
 \int_{\RR^d}\Big(I_{\Delta x}[ \tilde v_{i,k+1}](x_i- \Delta t a)+\Delta tL(x_i,a)\Big )\dd [ \mu_{i,k}-\tilde \mu_{i,k}](a)\geq 0 \quad{\text{for any $(i,k)\in \Ixt^*$}}.
\end{equation} 
For any  couple $w=(m,v) \in \S(\GG_{\Delta}) \times \BB(\GG_{\Delta})$,  we consider the operator $A_{\Delta }[w]$ defined as the set of values given by  
\begin{equation*}\label{eq:multivaluedDiscr}
	A_{\Delta }[w] \coloneqq \{ A_{i,k}[w] \, | \,  (i,k)\in\Ixt^* \},
\end{equation*}
where the map $A_{i,j}[w]$ is defined by 
\[
A_{i,k}[w]
\coloneqq
\begin{bmatrix}
	-v_{i,k} + I_{\Delta x}[v_{k+1}](x_i-\Delta t q_{i,k} ) + \Delta tL(x_i,q_{i,k}) +\Delta tf(x_i,m_k)\\
	m_{i,k+1} -  \sum_{j\in \Ix}m_{j,k}\int_{\RR^d}\beta_{i}(x_j-\Delta t a)\dd[\mu_{j,k}](a) \\
\end{bmatrix},
\]
with $q\in \Lambda_{\Delta}[v]$ and  $\mu\in \M_{\Delta}[v]$. We observe that, if $w=(m,v)$ is a solution to \eqref{eq:MFGsdiscr1}, then $A_{\Delta}[w]=0$.
For any $w=(m,v), \tilde w=(\tilde m,\tilde v) \in \S(\GG_{\Delta}) \times \BB(\GG_{\Delta})$, setting $q\in \Lambda_{\Delta}[v]$ and $\mu\in \M_{\Delta}[v]$,  we  define 
   \begin{equation*}
   \begin{aligned}
         &\langle  A_{\Delta }[ w], \tilde w \rangle_{\Delta }:= \\ &	 \sum_{k\in \It^*} \sum_{i\in\Ix} \Big(
		 -v_{i,k}  +  I_{\Delta x}[v_{k+1}](x_i - \Delta t q_{i,k}) + \Delta tL(x_i,q_{i,k}) +\Delta t f(x_i,m_k) \Big) (m_{i,k} - \tilde{m}_{i,k}) \\
		&+  \sum_{k\in \It^*} \sum_{i\in\Ix}\left( m_{i,k+1}  - \sum_{j\in\Ix} m_{j,k}\int_{\RR^d}\beta_i(x_j-\Delta t a)\dd[\mu_{j,k}](a) \right)(v_{i,k+1} - \tilde{v}_{i,k+1} ). 
   \end{aligned}
    \end{equation*}
For any $w,\ov w, \tilde w \in \S(\GG_{\Delta}) \times \BB(\GG_{\Delta})$, the following properties hold 
\begin{equation*}
\begin{aligned}
   \langle  A_{\Delta }[ w], \tilde w\rangle_{\Delta }-\langle  A_{\Delta }[\ov w], \tilde w\rangle_{\Delta }&=\langle  A_{\Delta }[ w]-A_{\Delta }[ \ov w], \tilde w\rangle_{\Delta } ,\\  \langle  A_{\Delta }[w], \tilde w\rangle_{\Delta }-\langle  A_{\Delta }[ w], \ov w\rangle_{\Delta }&=\langle  A_{\Delta }[ w], \tilde w-\ov w \rangle_{\Delta }.
\end{aligned}
\end{equation*}
The operator $A_{\Delta}$ satisfies the following monotonicity properties,  whose proof is given in \cref{Appendix}. 
\begin{prop}\label{Adelta properties}
   For any  $w=(m,v), \tilde w=(\tilde m,\tilde v) \in \S(\GG_{\Delta}) \times \BB(\GG_{\Delta})$ such that $m_k$ and $\tilde m_k$ have compact support for any $k\in \It$, setting $q\in \Lambda_{\Delta}[v]$,  $\tilde q\in \Lambda_{\Delta}[\tilde v]$, $\mu \in\M_{\Delta}[v]$ and $\tilde \mu \in\M_{\Delta}[\tilde v]$,   we have 
       \begin{equation}\label{eq:equality Adelta}
   \begin{aligned}
         \langle  A_{\Delta }[\tilde w]-A_{\Delta }[ w], \tilde w-w \rangle_{\Delta }&=  \sum_{i\in \Ix} \big((\tilde v_{i,0}-v_{i,0})(m_{i,0}-\tilde m_{i,0}) +(v_{i,N_T}-\tilde v_{i,N_T})(m_{i,N_T}-\tilde m_{i,N_T})\big)\\
+&\sum_{(i,k)\in\Ixt^*}m_{i,k}\int_{\RR^d}\Big(I_{\Delta x}[ \tilde v_{i,k+1}](x_i- \Delta t a)+\Delta tL(x_i,a)\Big )\dd [ \mu_{i,k}-\tilde \mu_{i,k}](a)\\ 
          +&\sum_{(i,k)\in\Ixt^*} m_{i,k}\int_{\RR^d}\Big(I_{\Delta x}[  v_{i,k+1}](x_i- \Delta t a)+\Delta tL(x_i,a)\Big )\dd[\tilde \mu_{i,k}-\mu_{i,k}](a)\\
            +&\sum_{(i,k)\in\Ixt^*}\Delta t \,\big(f(x_i,m_k)-f(x_i,\tilde m_k)\big)(m_{i,k}-\tilde m_{i,k}).\\ 
   \end{aligned}
    \end{equation}   
If \ref{DiscreteMonotonicity} is also satisfied, then 
           \begin{equation}\label{eq: inequality Adelta discrmon}
   \begin{aligned}
      \langle  A_{\Delta }[\tilde w]-A_{\Delta }[ w], \tilde w-w \rangle_{\Delta }&\geq   \sum_{i\in \Ix} \big((\tilde v_{i,0}-v_{i,0})(m_{i,0}-\tilde m_{i,0}) +(v_{i,N_T}-\tilde v_{i,N_T})(m_{i,N_T}-\tilde m_{i,N_T})\big).
   \end{aligned}
    \end{equation}
\end{prop}

\begin{rem}\label{rem: monotone descrete operator}
 Under \ref{DiscreteMonotonicity}, it follows that $A_{\Delta}$ combining with the initial and terminal conditions is \textit{monotone}, i.e. for any   $w=(m,v), \tilde w=(\tilde m,\tilde v) \in \S(\GG_{\Delta}) \times \BB(\GG_{\Delta})$ as in \cref{Adelta properties} and, in addition, such that,  for any $i\in \Ix$, $m_{i,0}=\tilde m_{i,0}$,    
 $v_{i,N_T}=g(x_i,m_{N_T})$  and $\tilde v_{i,N_T}=g(x_i,\tilde m_{N_T})$, 
 \begin{equation*}
      \langle  A_{\Delta }[\tilde w]-A_{\Delta }[ w], \tilde w-w \rangle_{\Delta }\geq 0.
 \end{equation*}
 \end{rem}
\begin{prop}\label{existenceMFGsd1}
     Let us assume \onlyref{continuity,initialdensity,ass:lagrangian} and $\Delta=(\Delta x,\Delta t)>0$ such that \ref{condition dxdt for compact support discrete solution} holds. Then there exists at least one solution to \eqref{eq:MFGsdiscr1}.
\end{prop}
\begin{proof}
We prove the existence of a solution to \eqref{eq:MFGsdiscr1} by applying
Kakutani’s fixed point theorem to a suitably defined multivalued map $\Psi:X\to 2^{X}\setminus \emptyset$,  as in  \cite{AshrafyanGomes}. We define the set $X\coloneqq\{\mu=(\mu_{i,k})_{(i,k)\in \Ixt^*}  \, | \, \mu_{i,k}\in \P_1(\RR^d), \, \supp(\mu_{i,k})\subseteq \ov B_\infty(0,C_H) \}$. It is compact and convex.
Given $\mu\in X$, let $m\in \S(\GG_{\Delta})$ be defined by  \eqref{eq:CEdiscr} and  $m_{\Delta}:=M_{\Delta}[m]$. Given $m_{\Delta}$,  let $v\in \BB(\GG_{\Delta})$ be defined by  \eqref{eq:HJBdiscr} and we take the set $\Lambda_{i,k}[v]$ defined by \eqref{optimal set}. Finally, we define 
\begin{equation*}
    \Psi(\mu)\coloneqq\{\ov \mu=(\ov \mu_{i,k})_{(i,k)\in \Ixt^*} \, | \, \ov \mu_{i,k}\in \P_1(\RR^d), \, \supp(\ov \mu_{i,k})\subseteq  \Lambda _{i,k}[v] \}.
\end{equation*}
We note that $\Psi(\mu)$ is a compact and convex set. The map  $\Psi$ is upper hemicontinuous, i.e. for any $\mu^n,\mu \in X$ such that $\mu^n \rightharpoonup \mu $ and for any $\ov \mu^n \in \Psi(\mu^n)$ such that  $\ov \mu^n \rightharpoonup \ov \mu$ for a some $ \ov \mu$, we have that $\ov \mu \in \Psi (\mu)$. To prove this property, we start by fixing $(i,k)\in \Ixt^*$.
According to \cite[Proposition 5.1.8]{AmbrosioGigliSavar}, the convergence $\ov \mu_{i,k} ^n \rightharpoonup \ov \mu_{i,k} $ implies that 
\begin{equation}\label{due to covnergennce of ovBn}
    \forall \ov a_{i,k} \in \supp(\ov \mu _{i,k}), \, \exists \,\ov a_{i,k}^n \in \supp(\ov \mu ^n_{i,k}) \, \text{ s.t. }\,  \lim_{n\to \infty}\ov a_{i,k}^n= \ov a_{i,k}. 
\end{equation}
As before, given $\mu^n$, let $m^n\in \S(\GG_\Delta)$ be defined by  \eqref{eq:CEdiscr} and  $m^n_{\Delta}:=M_{\Delta}[m^n]$. Given $m^n_{\Delta}$, let $v^n\in \BB(\GG_{\Delta})$ be defined by \eqref{eq:HJBdiscr}  and  $v^n_{\Delta}:=V_{\Delta}[v^n]$. In the same way, given $\mu$, we define $(m,m_{\Delta},v, v_{\Delta})$. We observe that $v^n_{\Delta}$ converges to $v_{\Delta}$ locally uniformly in $Q_T$. Indeed, according to  \cite[Proposition 7.1.5]{AmbrosioGigliSavar},  $\supp (\mu^n)_{i,k}\subseteq \ov B_\infty(0,C_H)$ and  $\mu ^n_{i,k} \rightharpoonup  \mu_{i,k}$ imply that  $\mu^n_{i,k}\to \mu_{i,k}$. Therefore, by definition of $m^n_{\Delta}$ and $m_{\Delta}$, we have $m^n_{\Delta}( t)\to m_{\Delta}(t)$ in $\P_1(\RR^d)$ for any $t \in [0,T]$. Combing with \ref{continuity}, this implies the locally uniform convergence of $v^n_{\Delta}$ to $v_{\Delta}$. We define $\omega^n,\omega: \ov B_{\infty}(0, C_H)\to \RR$ such that $ \omega^n(a)\coloneqq I_{\Delta x}[v^n_{k+1}](x_i-\Delta t a)+\Delta t L(x_i,a)$ and   $\omega(a)\coloneqq I_{\Delta x}[v_{k+1}](x_i-\Delta t a)+\Delta t L(x_i,a)$. We observe that $\omega^n$ converges uniformly to $\omega$ because $\sup_{a\in \ov B_{\infty}(0,C_H)}|\omega^n(a)-\omega(a)|=\sup_{a\in \ov B_{\infty}(0,C_H)}|v^n_{\Delta}(x_i-\Delta t a,t_k)-v_{\Delta}(x_i-\Delta t a,t_k)|$, due to \eqref{interperr}. 
Applying \cref{Compact nostrict Lemma 2.4 CapuzzoDolcetta}, we obtain that
\begin{equation}\label{due to covnergennce of controls}
\begin{aligned}
      &  \forall a^n_{i,k} \in \Lambda_{i,k}[v^n],  \text{ for any convergent subsequence (at least one exists), still denoted by $ a^n_{i,k}$}, \\ & \exists a^*_{i,k}\in  \Lambda_{i,k}[v] \,  \text{ s.t. } \,  \lim_{n\ \to \infty} a_{i,k}^{n}=  a^*_{i,k}.
\end{aligned}
\end{equation}
We aim to show that $\ov \mu \in \Psi(\mu)$,  i.e. for any $(i,k)\in \Ixt^*$ we have $\supp(\ov \mu_{i,k})\in \Lambda_{i,k}[v]$. Let $\ov a_{i,k}\in \supp(\ov \mu_{i,k})$. Due to \eqref{due to covnergennce of ovBn}, there exists  $\overline{a}^n_{i,k}\in \supp(\overline{\mu}_{i,k}^n)\subseteq\Lambda_{i,k}[v^n]$ s.t. $\lim_{n \to \infty}\ov a^n_{i,k}= \ov a_{i,k}$.   In addition, due to \eqref{due to covnergennce of controls}, there exists $ a^*_{i,k}\in  \Lambda_{i,k}[v]$ s.t. $\lim_{n\ \to \infty} \ov a_{i,k}^{n}=  a^*_{i,k}$ (where we passed to a further subsequence of $\{\overline{a}_{i,k}^n\}_{n}$ without changing the notation). For the uniqueness of the limit, $\ov a_{i,k}=a^*_{i,k} \in \Lambda_{i,k}[v]$. Finally, due to Kakutani's fixed-point theorem, there exists a fix point of the map $\Psi$.
\end{proof}
\cref{convergenceschemeMFGsD1} proves the convergence of the scheme \eqref{eq:MFGsdiscr1} to a weak solution to   the MFGs system, and implicitly also establishes the existence of such a solution. Its proof can be obtained using techniques analogous to those in  \cite[Proposition 7.7]{AshrafyanGomes}.
In the following, we will use {\em u.t.s.}
to indicate that the convergence holds up to subsequences.

\begin{theo}\label{convergenceschemeMFGsD1}
  Let us assume  \onlyref{continuity,C2bounded,initialdensity,ass:lagrangian}.  Consider a sequence  $\Delta_n=(\Delta x_n,\Delta t_n)>0$ such that $(\Delta x_n,\Delta t_n)\to 0$ and $\frac{\Delta x_n}{\Delta t_n}\to 0$ as $n\to \infty$, and assume that \onlyref{DiscreteMonotonicity,condition dxdt for compact support discrete solution} hold for any $(\Delta x_n, \Delta t_n)$.  Let $(m^n,v^n)$  be a solution to \eqref{eq:MFGsdiscr1} for the corresponding grid parameters $(\Delta x_n,\Delta t_n)$. Let $m_{\Delta_n}=M_{\Delta_n}[m^n]$ and $v_{\Delta_n}=V_{\Delta_n}[v^n]$.
Then there  exists  a weak solution $(m^*,v^*)$ to    the MFGs system such that, u.t.s., $m_{\Delta_n}(t)\to m^*(t)$ in $\P_1(\RR^d)$ for any $t\in [0,T]$,  and $v_{\Delta_n}\rightharpoonup^* v^*$ in $L^\infty(Q_T)$. 
\end{theo}
\begin{proof}

We define $w^n\coloneqq (m^n,v^n)$ and $w_{\Delta_n}\coloneqq(m_{\Delta_n},v_{\Delta_n})$.
Our goal is to prove that there exists a  $w^*=(m^*,v^*)$, weak u.t.s. limit of $w_{\Delta_n}$, 
such that
\begin{equation}\label{eq:tesi}
    \langle \mathcal A [ \varphi], \varphi-w^* \rangle \geq 0 \qquad \text{ for any } \varphi \in \mathcal{D}.
\end{equation}
 To prove \eqref{eq:tesi}, we will show that  for any $\varphi=(\phi,\psi)\in \mathcal{D}$ the following hold
 \begin{equation}\label{eq:tesi1}
\langle \mathcal A [ \varphi], \varphi-w_{\Delta _n} \rangle \to \langle \mathcal A [ \varphi], \varphi-w^* \rangle,\qquad  
 \end{equation}
 and 
 \begin{equation}\label{eq:tesi2}
     \mathcal{E}_1(\Delta_n)\leq \langle  A_{\Delta _n}[ \varphi^n], \varphi^n-w^n \rangle_{\Delta_n} = \langle \mathcal A [ \varphi], \varphi-w_{\Delta _n} \rangle + \mathcal{E}_2(\Delta_n)+\mathcal{E}_3(\Delta_n),
 \end{equation}
where  
   $\mathcal{E}_i(\Delta_n) \to 0$ as $n\to \infty$, for $i=1,2,3$,  $\varphi^n:=(\phi^n,\psi^n) \in \S(\GG_{\Delta_n})\times \BB(\GG_{\Delta_n})$ and
\begin{equation*}
    \phi^n_{i,k}\coloneqq \int_{E_i}\phi(x,t_k)\dd x, \qquad \psi^n_{i,k}\coloneqq \psi(x_i,t_k), \qquad \text{for any } (i,k)\in\Ixtn.
\end{equation*}
We define the  extensions $\phi_{\Delta_n}\coloneqq M_{\Delta_n}[\phi^n]$  and $ \psi_{\Delta_n}\coloneqq V_{\Delta_n}[\psi^n]$.
We observe that $\phi(\cdot,t)$, $\psi(\cdot,t)$, $\phi_{\Delta _n}(\cdot,t),\psi_{\Delta _n}(\cdot,t)$,
and $m_{\Delta_n}(\cdot,t)$
have  support in $\ov B(0,R_1)$, with $R_1 > C>0$ and $C$ as in \cref{SLCEproperties}.
Moreover, by the regularity of $\phi(\cdot,t_k)$, we have
\begin{equation}\label{eq:step5a}
    \sup_{k \in \Itn}d_1(\phi_{\Delta_n}(t_k),\phi(t_k))=\mathcal{O}(\Delta x_n),
\end{equation}
and  the following error estimate of the midpoint quadrature rule on $E_i$
\begin{equation}\label{eq:step5b}
     \sup_{(i,k)\in \Ixtn}|
     \phi_{\Delta_n}(x_i,t_k)-\phi(x_i,t_k)|
   =\mathcal{O}(\Delta x_n^{2}).
\end{equation}
We organize the proof in the following steps, where we will make us of the following definitions.
We set
 $\Jxn\coloneqq\{i\in\Ixn | x_i\in \ov B(0,R_1+\sqrt{d}TC_H) \}$,
which implies $\sum_{i \in \Jxn}\beta_i (x_j- \Delta t_n a)=1$, for any   $a \in \ov B_{\infty}(0,C_H)$ and  $x_j \in \ov B (0,R_1)$. Given $\eta^n\in\M_{\Delta_n}[\psi^n]$, we define $\eta_{\Delta_n}^{x,t}\in \P(\RR^d)$ as 
\begin{equation*}
\begin{aligned}
\eta_{\Delta_n}^{x,t}\coloneqq \eta^n_{i,k} \quad \text{if }(x,t) \in E_i\times[t_k,t_{k+1}),\quad \eta_{\Delta_n}^{x,t}\coloneqq \eta^n_{i,N_{T}-1} \quad \text{if }t=T,
\end{aligned}
\end{equation*}
 for any $(x,t)\in Q_T$.

\textit{Step 1: Existence of $w^*$ and proof of \eqref{eq:tesi1}.} By \cref{properties decrete value function}, $||v_{\Delta_n}||_{L^\infty(Q_T)}\leq  \widetilde C_{\text{{\rm v}}}$ for any $n\in\NN$. According to the Riesz Representation theorem, there exists $v^*\in L^\infty(Q_T)$ such that, u.t.s., 
 $v_{\Delta_n}\rightharpoonup^* v^*$,  
\begin{equation*}
 {\text {i.e.}}\quad  \int_{Q_T}v_{\Delta_n}(x,t) \theta(x,t)\dd x \dd t\to \int_{Q_T}v^*(x,t) \theta(x,t)\dd x\dd t \qquad \text{for any } \theta \in L^1(Q_T).
\end{equation*}
For any $t\in [0,T]$, by \cref{SLCEproperties}, $m_{\Delta_n}(t)$ are uniform tight measures. Then, by Prokhorov's theorem (see \cite[Theorem 5.1.3]{AmbrosioGigliSavar}), the sequence $m_{\Delta_n}(t)$ is relative compact in $\P_1(\RR^d)$. By \cref{SLCEproperties}, $m_{\Delta_n}(t)$ have uniformly integrable first moments. Then, by   \cite[Proposition 7.1.5]{AmbrosioGigliSavar}, there exists $m^*(t)\in \P_1(\RR^d)$ such that, u.t.s., $m_{\Delta_n}(t) \to m^*(t)$ in $\P_1(\RR^d)$.
 In addition, since $||m_{\Delta_n}(\cdot,t)||_{L^1(\RR^d)}=1$ for any $t\in [0,T]$, $m_{\Delta_n}(\cdot,t)\rightharpoonup m^*(\cdot,t)$ in $L^1(\RR^d)$. We observe that $m^*(x,t)=0$  for a.e.   $x \notin\ov B(0,C)$, and we call $w^*\coloneqq(m^*,v^*)$. This also proves \eqref{eq:tesi1}.
 
\textit{Step 2: Definition of  $\mathcal{E}_1(\Delta_n)$ and proof that $\mathcal{E}_1(\Delta_n)\to 0$.} We start by noticing that $\langle  A_{\Delta _n}[ \varphi^n], \varphi^n-w^n \rangle_{\Delta_n} = \langle  A_{\Delta _n}[ \varphi^n]-A_{\Delta _n}[ w^n], \varphi^n-w^n \rangle_{\Delta_n}$, since $A_{\Delta_n}[w^n]=0$. Applying \eqref{eq: inequality Adelta discrmon} combined with  the initial and terminal conditions in \eqref{eq:MFGsdiscr1},
we obtain 
\begin{equation*}
    \langle  A_{\Delta _n}[ \varphi^n], \varphi^n-w^n \rangle_{\Delta_n}\geq   \sum_{i\in \Ixn}(g(x_i,m^n_{N_T})- G(x_i,\phi(T))(m^n_{i,N_T}-\phi^n_{i,N_T}). 
\end{equation*}
By adding and subtracting  the terms $g(i,\phi^n_{N_T})(m^n_{i,N_T}-\phi^n_{i,N_T})$ for any $i\in \Ixn$ and using the discrete monotonicity \cref{DiscreteMonotonicity} of $g$, we further 
 obtain 
\begin{equation*}
    \langle  A_{\Delta _n}[ \varphi^n], \varphi^n-w^n \rangle_{\Delta_n}\geq   \sum_{i\in \Ixn}(g(x_i,\phi^n_{N_T})-  G(x_i,\phi (T)))(m^n_{i,N_T}-\phi^n_{i,N_T}). 
\end{equation*}
 Therefore, we set $\mathcal{E}_1(\Delta_n)\coloneqq \sum_{i\in \Ixn}(g(x_i,\phi^n_{N_T})-  G(x_i,\phi (T)))(m^n_{i,N_T}-\phi^n_{i,N_T})$.  We estimate $\mathcal{E}_1(\Delta_n)$ from above, noticing that $g(i,\phi^n_{N_T})=\frac{1}{\Delta x_n^d}\int_{E_i}G(x,\phi_{\Delta_n}(T)\dd x)$, 
\begin{equation*}
\begin{aligned}  
&|\mathcal{E}_1(\Delta x_n,\Delta t_n)|=\Big |\sum_{i\in \Jxn}\Big ( \frac{1}{\Delta x_n^d}\int_{E_i}G(x,\phi_{\Delta_n}(T))\dd x-  G(x_i,\phi (T))\Big)(m^n_{i,N_T}-\phi^n_{i,N_T})\Big|\\
& \leq \Big |\sum_{i\in \Jxn} (m^n_{i,N_T}-\phi^n_{i,N_T})\frac{1}{\Delta x_n^d}\int_{E_i}\big (G(x,\phi_{\Delta_n}(T))-G(x_i,\phi_{\Delta_n}(T))\big)\dd x \Big|
\\ & +|\sum_{i\in \Jxn}(m^n_{i,N_T}-\phi^n_{i,N_T}) \big (G(x_i,\phi_{\Delta_n}(T))-G(x_i,\phi(T))\big )|\\ 
 &\leq \sqrt{d}C_{G,2}\Delta x_n+2 \sup_{x\in\ov B(0,R_1)}|G(x,\phi_{\Delta_n}(T))- G(x,\phi (T))|.\\ 
\end{aligned}
\end{equation*}
We consider the sequence of functions $G(\cdot,\phi_{\Delta_n}(T)):\ov B(0,R_1)\to \RR$. By \eqref{eq:step5a} and \ref{continuity}, this sequence converges pointwise to  $G(\cdot,\phi(T))$. Moreover, by \ref{C2bounded},  these functions  are equi-bounded and equi-Lipschitz. By Ascoli-Arzelà theorem, there exists a subsequence, still denoted by $G(\cdot,\phi_{\Delta_n}(T))$, that converges uniformly to $G(\cdot,\phi(T))$. 
Since this limit is unique, we  conclude that  $\lim_{n\to \infty}\sup_{x\in\ov B(0,R_1)}|G(x,\phi_{\Delta_n}(T))- G(x,\phi (T)|= 0$, which implies that $\mathcal{E}_1(\Delta_n)\to 0$ as $n\to \infty$.

\textit{Step 3: Definition of  $\mathcal{E}_2(\Delta_n)$ and proof that $\mathcal{E}_2(\Delta_n)\to 0$.} We define 
\begin{equation*}
 \begin{aligned}
          &\mathcal{E}_2(\Delta_n)\coloneqq\int_{Q_T}\big(\partial_t\psi(x,t)-H(x,D_x\psi(x,t))+F(x,\phi(t)\big)\big(m_{\Delta_n}(x,t)-\phi(x,t)\big)\dd x\dd t \\ 
          &+\sum_{k\in \Itns} \sum_{i\in\Ixn} \Big(
		  \psi^n_{i,k}   - I_{\Delta x_n}[ \psi^n_{k+1}](x_i - \Delta t_n \xi^n_{i,k}) - \Delta t_n L(x_i,\xi^n_{i,k})  -\Delta t_n  f(x_i,\phi^n_k) \Big) (m^n_{i,k} - \phi^n_{i,k}). 
 \end{aligned}
 \end{equation*}
By defining the grid function $\zeta^n \coloneqq m^n-\phi^n$ and the functions 
$h\coloneqq\partial_t\psi-H(x,D_x\psi)+F(x,\phi(t))$, $\zeta_{\Delta_n}=m_{\Delta_n}-\phi_{\Delta_n}$,  we can rewrite $\mathcal{E}_2(\Delta_n)$ as the sum of the following four terms.
The first term is
\begin{equation}\label{err2: term1}
    \Delta t_n \sum_{k\in \Itns} \sum_{i\in\Ixn} \Big(
		\frac{1}{\Delta x_n^d}  \int_{E_i}F(x,\phi(t_k))\dd x - f(x_i,\phi^n_k) \Big) \zeta^n_{i,k}
\end{equation}
and tends to zero, arguing as in {\em Step 2}. The second term is\begin{equation}\label{err2: term2}
\begin{aligned}
          &\Delta t_n\sum_{k\in \Itns} \sum_{i\in\Ixn} \frac{
		  \psi^n_{i,k}   - I_{\Delta x_n}[ \psi^n_{k+1}](x_i - \Delta t_n \xi^n_{i,k}) - \Delta t_n L(x_i,\xi^n_{i,k})  -\frac{\Delta t_n}{\Delta x_n^d}  \int_{E_i}F(x,\phi(t_k))\dd x} {\Delta t_n}\zeta^n_{i,k} \\ 
  &-\Delta t_n \sum_{k\in \Itns} \sum_{i\in\Ixn} \big(\partial_t\psi(x_i,t_k)-H(x_i,D_x\psi(x_i,t_k))+F(x_i,\phi(t_k))\big)\zeta^n_{i,k},
		\\
\end{aligned}
\end{equation}
and goes to zero by the consistency property in \cref{properties decrete value function}. The third term is 

\begin{equation}\label{err2: term3}
    \begin{aligned}
    &\int_{Q_T}h(x,t)\zeta_{\Delta_n}(x,t)\dd x\dd t-\Delta t_n \Delta x_n^d\sum_{(i,k)\in \Ixtn^*} h(x_i,t_k)\zeta_{\Delta_n}(x_i,t_k),
    \end{aligned}
\end{equation}
and goes to zero because,  using that $\supp( \zeta_{\Delta_n}(\cdot,t)) \subseteq \ov B(0,R_1)$ for any $t \in [0,T]$, we can write 
\begin{equation*}
\begin{aligned}
   & \int_{Q_T}h(x,t)\zeta_{\Delta_n}(x,t)\dd x\dd t=\int_{\ov B(0,R_1)}\sum_{k\in \Itns}\int_{t_k}^{t_{k+1}}h(x,t)\zeta_{\Delta_n}(x,t_k)\dd t\dd x \\ &   =  \Delta t_n \sum_{k\in\Itns} \int_{\ov B(0,R_1)} h(x,t_k)\zeta_{\Delta_n}(x,t_k)\dd t\dd x+\mathcal{O}(\Delta t_n)\\
    & =  \Delta x_n^d \Delta t_n \sum_{k\in\Itns}  \sum_{i\in\Ixn}  h(x_i,t_k)\zeta_{\Delta_n}(x_i,t_k)+\mathcal{O}(\Delta t_n).\\
\end{aligned}
\end{equation*}
Finally, the fourth term is
\begin{equation}\label{err2: term4}
    \int_{Q_T}h(x,t)\left(\phi(x,t)-\phi_{\Delta_n}(x,t)\right)\dd x\dd t,
\end{equation}
and goes to zero by \eqref{eq:step5a}. 

\textit{Step 4: Convergence of the optimal controls.}  Given $(x,t)\in Q_T$, we consider a sequence of grid points $(x_{i_n},t_{k_n})\in \Ixtn^*$ such that
$(x_{i_n},t_{k_n})\to(x,t)$.  For any $\xi^n_{i_n,k_n}\in \Lambda_{i_n,k_n}[\psi^n]$, we have  $\xi^n_{i_n,k_n}\to D_pH(x,D_x\psi(x,t))$ as $n\to \infty$. Indeed, by defining 
$\omega^n,\omega: \ov B_{\infty}(0, C_H)\to \RR$ as
\begin{equation*}
    \omega^n(a)\coloneqq \frac{I_{\Delta x_n}[\psi^n_{k_n+1}](x_{i_n}-a_n\Delta t_n )-\psi(x_{i_n},t_{k_n})}{\Delta t_n}+ L(x_{i_n},a)
\end{equation*} 
and $ \omega(a)\coloneqq -aD_x\psi (x,t)+L(x,a)$,  it follows that 
$\omega^n,\omega\in \CC(\ov B_{\infty}(0,C_H))$ and $\omega^n$ converges uniformly to $\omega$.  
According to \cref{Compact Lemma 2.4 CapuzzoDolcetta}, we conclude that   $\xi^n_{i_n,k_n}\to D_pH(x,D_x\psi(x,t))$ as $n\to \infty$, since $D_pH\big(x,D_x\psi (x,t)\big)=\amin_{a\in \ov{B}_{\infty}(0,C_H)} \omega(a)$. 

\textit{Step 5: Convergence of the relaxed controls.} For any  $(x,t)\in Q_T$, we have that  $\eta_{\Delta_n}^{x,t}\to \delta_{D_pH(x,D_x\phi(x,t))}$ in $\P_1(\RR^d)$ as $n\to \infty$. Indeed, since by definition $\supp(\eta_{\Delta_n}^{x,t})\subseteq \ov B_{\infty}(0,C_H)$, $\eta_{\Delta_n}^{x,t}$ is a sequence of uniform tight measures with uniformly integrable first moments. Then, by Prokhorov's theorem (see \cite[Theorem 5.1.3]{AmbrosioGigliSavar}), $\eta_{\Delta_n}^{x,t}$ is relative compact in $\P_1(\RR^d)$ and by \cite[Proposition 7.1.5]{AmbrosioGigliSavar}, we conclude that there exists $\eta^{x,t}\in \P_1(\RR^d)$ such that, u.t.s., $\eta_{\Delta_n}^{x,t}\to \eta^{x,t}$. It remains to show that $\eta^{x,t}=\delta_{D_pH(x,D_x\psi(x,t))}$, i.e. $\supp(\eta^{x,t})=\{D_pH(x,D_x\psi(x,t))\}$. 
By definition, for any $n\in \NN$,  there exists a sequence of grid points $(i_n,k_n)\in \Ixtn^*$ such that $(x,t)\in E_{i_n}\times [t_{k_n},t_{k_n+1})$. Then, $\eta_{\Delta_n}^{x,t}=\eta^n_{i_n,k_n}$ and  $\supp(\eta_{\Delta_n}^{x,t})\subseteq\Lambda_{i_n,k_n}[\psi^n]$. By \cite[Proposition 5.1.3]{AmbrosioGigliSavar}, for any $a\in \supp(\eta^{x,t})$, there exists a sequence $a^n\in \supp(\eta_{\Delta_n}^{x,t})$ such that $a^n\to a$. By \emph{Step 4},  $a=D_pH(x,D_x\phi(x,t)).$ Finally, since the limit is unique, then the whole sequence converges to $\delta_{D_pH(x,D_x\psi(x,t))}$. 

\textit{Step 6:  Definition of  $\mathcal{E}_3(\Delta_n)$ and proof that $\mathcal{E}_3(\Delta_n)\to 0$.}  We define
 \begin{equation*}
 \begin{aligned}
          \mathcal{E}_3(\Delta_n):=  &\int_{Q_T}\Big(\partial_t\phi(x,t)-\diver\big(\phi(x,t)D_pH\big(x,D_x\psi(x,t)\big)\big)\Big)\big(v_{\Delta_n}(x,t)-\psi(x,t)\big)\dd x\dd t\\
          &+\sum_{k\in \Itns} \sum_{i\in\Ixn} \Big(- \phi^n_{i,k+1}  + \sum_{j\in\Ixn}  \phi^n_{j,k} \int_{\RR^d}\beta_i(x_j-\Delta t_n a )\dd[\eta^n_{j,k}](a) \Big) (v^n_{i,k+1} - \psi^n_{i,k+1} ).       
 \end{aligned}
 \end{equation*}
We define  the grid function   $l^n:= v^n-\psi^n$ and the function $l_{\Delta_n}\coloneqq v_{\Delta_n}-\psi$. We rewrite $\mathcal{E}_3(\Delta_n)$ as the sum of the following three terms.
The first term of $\mathcal{E}_3(\Delta_n)$ is
\begin{equation}\label{err3: term1}
    \begin{aligned}
       &  \Delta t_n \sum_{k\in \Itns} \sum_{i\in\Ixn} \left(- \frac{\phi^n_{i,k+1}-\phi^n_{i,k}}{\Delta t_n} \right) l_{\Delta_n}(x_i,t_{k+1}) +\int_{Q_T}\partial_t\phi(x,t)l_{\Delta_n}(x,t)\dd x\dd t\\
    \end{aligned}
\end{equation}
and it goes to zero. Indeed, 
 the regularity of $\psi$ and $\phi$, the definition of $R_1$, the Lipschitz property of $l_{\Delta_n}(\cdot, t)$ for any $t \in [0,T]$ (see \cref{properties decrete value function}),  and  \eqref{eq:step5b},  imply that the second term in  \eqref{err3: term1} is such that
\begin{equation*}
    \begin{aligned}
        &\sum_{k\in \Itns} \int_{\ov B(0,R_1)}\int_{t_k}^{t_{k+1}}\partial_t\phi(x,t)\big(v_{\Delta _n}(x,t_{k+1})-\psi(x,t)\big)\dd x\dd t \\
                & =\Delta t_n\sum_{k\in \Itns} \int_{\ov B(0,R_1)}\partial_t\phi(x,t_{k+1})l_{\Delta_n}(x,t_{k+1})\dd x+\mathcal{O}(\Delta t_n)\\
                   & =\Delta t_n\sum_{k\in \Itns} \int_{\ov B(0,R_1)}  \frac{\phi(x,t_{k+1})-\phi(x,t_{k})}{\Delta t_n}l_{\Delta_n}(x,t_{k+1})\dd x +\mathcal{O}(\Delta t_n)\\
       &=\Delta x_n^d \Delta t_n \sum_{k\in \Itns} \sum_{i\in\Jxn}\frac{\phi(x_i,t_{k+1})-\phi(x_i,t_{k})}{\Delta t_n}  l_{\Delta_n}(x_i,t_{k+1})  +\mathcal{O}\left(\Delta t_n+\frac{\Delta x_n}{\Delta t_n}\right)\\
          & = \Delta t_n \sum_{k\in \Itns} \sum_{i\in\Jxn} \frac{\phi^n_{i,k+1}-\phi^n_{i,k}}{\Delta t_n}  l_{\Delta_n}(x_i,t_{k+1}) +\mathcal{O}\left(\Delta t_n+\frac{\Delta x_n}{\Delta t_n}\right).
 \end{aligned}
 \end{equation*}
Where the term $\frac{\Delta x_n}{\Delta t_n}$  appears because we are approximating the integral in space of the function $\frac{\phi(x,t_{k+1})-\phi(x,t_{k})}{\Delta t_n}l_{\Delta_n}(x,t_{k+1})$ that is only Lipschitz continuous in space with Lipschitz constant of order $\frac{1}{(\Delta t_n)}$.
The second term of $\mathcal{E}_3(\Delta_n)$ is
\begin{equation}\label{err3: term2}
\begin{aligned}
&\sum_{k\in \Itns} \sum_{j\in\Ixn}\phi^n_{j,k}\Big(\sum_{i\in\Ixn} \beta_i(x_j-\Delta t_nD_pH(x_j,D_x \psi(x_j,t_{k+1})) )  l^n_{i,k+1}  -l^n_{j,k+1} \Big)\\ 
&-\int_{Q_T}\diver\big(\phi(x,t)D_pH(x,D_x \psi(x,t))\big)l_{\Delta_n}(x,t)\dd x\dd t,      
\end{aligned}
\end{equation}
and it goes to zero. Indeed, the definition of $R_1$, the regularity of $\phi$ and $\psi$, estimates \eqref{interperr} and \eqref{eq:step5b},
and the Lipschitz property of $l_{\Delta_n}(\cdot, t)$ for any $t \in [0,T]$, imply
that the first term in \eqref{err3: term2} is such that
\begin{equation*}\label{err3: term2 sum scrittura1}
    \begin{aligned}
      &\sum_{k\in \Itns} \sum_{j\in\Jxn}\phi^n_{j,k}\Big(\sum_{i\in\Jxn} \beta_i(x_j-\Delta t_nD_pH(x_j,D_x \psi(x_j,t_{k+1})) )  l^n_{i,k+1}  -l^n_{j,k+1} \Big)\\ 
     &=  \sum_{k\in \Itns} \sum_{j\in\Jxn}\phi^n_{j,k}\Big(I_{\Delta x_n}[l^n_{k+1}](x_j-\Delta t_n D_pH(x_j,D_x \psi(x_j,t_{k+1})))-l^n_{j,k+1}\Big) \\ 
    & = \sum_{k\in \Itns} \sum_{j\in\Jxn} \phi^n_{j,k}\Big( l_{\Delta_n}(x_j-\Delta t_n D_pH(x_j,D_x \psi(x_j,t_{k+1})),t_{k+1}) -l^n_{j,k+1}  \Big)  +\mathcal{O}\left(\frac{\Delta x_n^2}{\Delta t_n}\right) \\
    &        
      =\Delta x_n^d \sum_{\substack{k\in \Itns \\j\in\Jxn
      } }\phi(x_j,t_k)\Big( l_{\Delta_n}(x_j-\Delta t_n D_pH(x_j,D_x \psi(x_j,t_{k+1})) ,t_{k+1}) -l^n_{j,k+1} \Big) +\mathcal{O}\left(\frac{\Delta x_n^2}{\Delta t_n}\right) \\ 
    &= \int_{\ov B(0,R_1)\times [0,T]}\phi(x,t)  \frac{l_{\Delta_n}(x-\Delta t_nD_pH(x,D_x \psi (x,t)) ,t)-l_{\Delta_n}(x,t)}{\Delta t_n} \dd x\dd t  +\mathcal{O}\left(\frac{\Delta x_n}{\Delta t_n}+\Delta t_n\right).
    \end{aligned}
\end{equation*}
Moreover observing that,  again by the Lipschitz property of $l_{\Delta_n}(\cdot, t)$, 
\begin{equation*}
    \begin{aligned}
         & \int_{\ov B(0,R_1)\times [0,T]}\phi(x,t)  \frac{l_{\Delta_n}(x-\Delta t_nD_pH(x,D_x \psi (x,t)) ,t)-l_{\Delta_n}(x,t)}{\Delta t_n} \dd x\dd t  \\ 
         &=\int_{\ov B(0,R_1)\times[0,T]}\phi(x,t)D_pH(x,D_x \psi(x,t))\cdot D_x l_{\Delta_n}(x,t)\dd x\dd t +o(1),
    \end{aligned}
\end{equation*}
using integration by part and recalling the definition  of $R_1$, we conclude that \eqref{err3: term2} goes to 0.
The third term is
\begin{equation}\label{err3: term3}
\begin{aligned}
  &\sum_{k\in \Itns} \sum_{i\in\Ixn} \sum_{j\in\Ixn}  \phi^n_{j,k}   l_{\Delta_n}(x_i,t_{k+1})\int_{\RR^d}\beta_i(x_j-\Delta t_n a )\dd [\eta^n_{j,k}](a) \\
        &-\sum_{k\in \Itns} \sum_{i\in\Ixn} \sum_{j\in\Ixn} \beta_i(x_j-\Delta t_nD_pH(x_j,D_x \psi (x_j,t_{k+1})) ) \phi^n_{j,k}  l_{\Delta_n}(x_i,t_{k+1}) 
\end{aligned}
\end{equation}
and it goes to zero.
Indeed, by
 the definition of $R_1$, by the regularity of $\psi$ and$\phi$, by the estimate \eqref{interperr}, by the Lipschitz property of $l_{\Delta_n}(\cdot, t)$ (we denote  $\tilde C$ its Lipschitz constant),  and  by \eqref{eq:step5b}  
 we can approximate \eqref{err3: term3} as following,
\begin{equation*}\label{err3: term3 in last step}
    \begin{aligned}
         & \sum_{k\in \Itns} \sum_{j\in\Jxn}\phi^n_{j,k} \int_{\RR^d}I_{\Delta x_n}[ l^n_{k+1}](x_j-\Delta t_n a)\dd[\eta^n_{j,k}](a)\\ 
        & - \sum_{k\in \Itns} \sum_{j\in\Jxn}\phi^n_{j,k} I_{\Delta x_n}[l^n_{k+1}](x_j-\Delta t_nD_pH(x_j,D_x \psi(x_j,t_{k+1})))\\
   &= \sum_{\substack{k\in \Itns \\j\in\Jxn
      } }\phi^n_{j,k} \int_{\RR^d}l_{\Delta_n}(x_j-\Delta t_n a,t_{k+1})-l_{\Delta_n}(x_j-\Delta t_nD_pH(x_j,D_x \psi(x_j,t_{k+1})),t_{k+1})\dd[\eta^n_{j,k}](a)\\
      &+ \mathcal{O}\left(\frac{\Delta x_n^2}{\Delta t_n}\right),
    \end{aligned}
 \end{equation*}  
 where the term
 \begin{equation*}
       \begin{aligned}  
        &\Big|\sum_{\substack{k\in \Itns \\j\in\Jxn
      } }\phi^n_{j,k} \int_{\RR^d} l_{\Delta_n}(x_j-\Delta t_n a,t_{k+1})-l_{\Delta_n}(x_j-\Delta t_nD_pH(x_j,D_x \psi(x_j,t_{k+1})),t_{k+1})\dd[\eta^n_{j,k}](a)\Big|\\
        &\leq \tilde C \Delta t_n \sum_{k\in \Itns} \sum_{j\in\Jxn}\phi^n_{j,k} \int_{\RR^d}|a-D_pH(x_j,D_x\psi (x_j,t_{k+1}))|\dd[\eta^n_{j,k}](a)\\
        &=\tilde C\Delta x_n^d \Delta t_n \sum_{k\in \Itns} \sum_{j\in\Jxn}\phi(x_j,t_k) \int_{\RR^d}|a-D_pH(x_j,D_x\psi (x_j,t_{k+1}))|\dd[\eta^n_{j,k}](a)+\mathcal{O}\left(\Delta x^2_n\right) \\
          &=\tilde C \int_{0}^{T}\int_{\ov B(0,R_1)}\phi(x,t) \int_{\RR^d}|a-D_pH(x,D_x\psi (x,t))|\dd[\eta_{\Delta_n}^{x,t}](a)\dd x \dd t+\mathcal{O}\left(\Delta x_n+\Delta t_n\right)
    \end{aligned}
\end{equation*}
goes to zero by the Dominated convergence theorem and {\textit{Step 5}}.
In conclusion, $\mathcal{E}_3({\Delta_n})\to 0$, since \eqref{err3: term1}, \eqref{err3: term2} and \eqref{err3: term3} go to zero.

\end{proof}
\begin{rem}
Let us note that an inverse CFL condition $\frac{\Delta x}{\Delta t}\to 0$ is required to ensure convergence,  as is typical in SL scheme,  allowing for large time steps.
A similar condition is required in
\cref{properties decrete value function} to guarantee the consistency of the scheme for the value function, and  in \cref{SLCEproperties}   to ensure the compact support of the discrete distribution.
\end{rem}

\begin{rem}
We are not able to prove that the pair $(m^*,v^*)$ is a distributional-viscosity solution to the MFGs system as in \cite{CarliniSilva2014}, since proving that  $m^*$ is a distributional solution to the continuity equation relies on specific properties of the regularized gradient. These properties do not generally hold in our setting, since the drift is determined by discrete optimal controls.
However,  arguing as in \cite[Proposition 3.9]{CarliniSilva2014}, the boundedness of the discrete controls implies uniform equi-continuity in time of $m_{\Delta_n}$, which,   combined  with   Lemma \ref{SLCEproperties}, allows us to prove that $m_{\Delta_n} \to   m^* $ in $ \CC([0,T];\P_1(\RR^d))$.  Moreover, arguing as in \cite[Theorem 3.3]{CarliniSilva2014}, we  have $v_{\Delta_n}\to v^*$ uniformly over compact sets of $\RR^d\times [0,T]$, and  $v^*$ is the viscosity solution to the HJB equation given the measure $m^*$.
\end{rem}

 \cref{rem: discrete monotonicity} and \cref{convergenceschemeMFGsD1} prove the following  result. 
\begin{cor}\label{existencewealsolution}
Let us assume  \onlyref{continuity,C2bounded,initialdensity,ass:lagrangian} and \ref{monoticityassumption}. Then there exists a weak solution to   the MFGs system.
\end{cor}

\subsection{A semi-Lagrangian scheme \eqref{eq:MFGsdiscr2}  for MFGs}\label{sec: A semi-Lagrange scheme 2 for MFGs}
In the reminder of the article, we suppose verified the following assumption:
\begin{Hassum}[resume]
  \item \label{uniquecontrol}
For every $v\in \BB(\GG_{\Delta x})$, $(i,k)\in \Ixt^*$, the set of {discrete optimal controls} $\Lambda_{i,k}[v]$
is a singleton.
\end{Hassum}

Although the discrete optimal controls are not necessarily unique in general, the assumption \ref{uniquecontrol} of uniqueness of the minimizer is reasonable, as it is frequently observed in practice. 
   The authors of \cite{AshrafyanGomes} note that the minimizer is typically unique, except in particular cases such as those involving symmetry. This empirical observation suggests that, in most cases, the function to be minimized (the discrete cost) has a single minimum point.

Moreover, the uniqueness assumption is a key ingredient in proving the convergence of the  Learning Value Iteration algorithm, which will be introduced and analyzed later in the paper. In addition, it simplifies the implementation: in this case, the relaxed control reduces to a Dirac measure supported on the unique optimal control. Without uniqueness, one would need to compute and select among multiple relaxed controls, which would introduce additional computational cost and ambiguity. Therefore, assuming uniqueness 
can be seen as a practical simplification that enhances computational efficiency. 

We consider the following simplified SL scheme to the MFGs system: find {$(m,v) \in \S(\GG_{\Delta})\times \BB(\GG_{\Delta})$ such that 
\begin{equation}\label{eq:MFGsdiscr2}\tag{$MFG_d^2$}
            \begin{aligned}
               &v_{i,k}=I_{\Delta x}[v_{k+1}](x_{i}-\Delta t q_{i,k})+\Delta tL(x_i,q_{i,k})  +\Delta t f(x_i,m_k)&&\text{for any   }(i,k)\in \Ixt^*, \\
                 &m_{i,k+1}=\sum_{j\in \Ix}\beta_{i}(x_{j}-\Delta tq_{j,k})m_{j,k}  &&\text{for any   }(i,k)\in \Ixt^*,\\
               &v_{i,N_T}=g (i, m_{N_T}), \qquad m_{i,0}=\int_{E_{i}}\dd [m^*_{0}](x)&&  \text{for any  }i\in \Ix \\  
            \end{aligned}
\end{equation}
 and $q \in \BB(\GG^*_{\Delta};\RR ^d)$ such that
\begin{equation}\label{eq:qMFGs}
        q_{i,k}=\amin_{a\in \ov{B}_{\infty}(0,C_H)}\left\{I_{\Delta x}[v_{k+1}](x_{i}-\Delta t a)+\Delta tL(x_i,a)\right\}\qquad \text{for any   }(i,k)\in \Ixt^*.
\end{equation}
\begin{rem}\label{rem equivalent schemes}
     We notice that, under \ref{uniquecontrol},  the schemes \eqref{eq:MFGsdiscr1} and \eqref{eq:MFGsdiscr2} are equivalent. In fact, \ref{uniquecontrol} implies that there exists a unique $q \in \Lambda_{\Delta}[v]$. Consequently, if $\mu \in \M_{\Delta}[v]$, then $\mu_{i,k}=\delta_{q_{i,k}}$ for any $(i,k)\in \Ixt^*.$ 
\end{rem}
Using the monotonicity property \cref{Adelta properties}, we can prove a uniqueness result as follows:\begin{prop}\label{uniquenessMFGsd2}
    Let us assume \ref{initialdensity} 
    and let $\Delta=(\Delta x,\Delta t)>0$ be such that \onlyref{uniquecontrol,DiscreteMonotonicity,condition dxdt for compact support discrete solution} hold. If $(m,v)$ and $(\tilde m,\tilde v)$  are two solutions to \eqref{eq:MFGsdiscr2}, then $(m,v)=(\tilde m,\tilde v)$.
\end{prop}
\begin{proof}
     The proof can be obtained similarly to \cite[Lemma 7.2]{AshrafyanGomes}. 
Let $w=(m,v)$ and $\tilde w=(\tilde m,\tilde v)$  be two solutions to \eqref{eq:MFGsdiscr2}, $q\in \Lambda_{\Delta}[v]$ and $\tilde q\in \Lambda_{\Delta}[\tilde v]$ be the respective unique grid functions of optimal controls, and $\mu\in \M_{\Delta}[v]$ and $\tilde \mu\in \M_{\Delta}[\tilde v]$ be the relaxed controls as in \cref{rem equivalent schemes}. Then $A_\Delta [w]=A_\Delta [\tilde w]$. According to \cref{SLCEproperties}, both $m_k$ and $\tilde m_k$ have compact support for any $k\in \It$. Using \eqref{eq:equality Adelta}, the initial and terminal condition in \eqref{eq:MFGsdiscr2}, \ref{DiscreteMonotonicity},  \eqref{minproperty} and the fact that $m_{i,k}, \tilde m_{i,k}\geq 0$ for any $(i,k)\in \Ixt$, we obtain 
       \begin{equation*}
        \begin{aligned}
           &0 = \langle  A_\Delta[\tilde w]-A_\Delta[ w], \tilde w-w \rangle_{\Delta} \\ 
            & = \sum_{i\in \Ix}\big(g(i,m_{N_T})-g(i,\tilde m_{N_T})\big)(m_{i,N_T}-\tilde m_{i,N_T}) \\
        &+\sum_{(i,k)\in\Ixt^*}m_{i,k}\big(I_{\Delta x}[ \tilde v_{i,k+1}](x_i- \Delta t q_{i,k})+\Delta tL(x_i,q_{i,k})-I_{\Delta x}[\tilde v_{i,k+1}](x_i-\Delta t \tilde q_{i,k})-\Delta tL(x_i,\tilde q_{i,k}) \big)\\ 
        &+\sum_{(i,k)\in\Ixt^*}\tilde m_{i,k}\big(I_{\Delta x}[  v_{i,k+1}](x_i-  \Delta t \tilde q_{i,k})+\Delta tL(x_i,\tilde q_{i,k})-I_{\Delta x}[v_{i,k+1}](x_i- \Delta t q_{i,k})-\Delta tL(x_i,q_{i,k})\big)\\
            &+\Delta t\sum_{(i,k)\in\Ixt^*}\big(f(x_i,m_k)-f(x_i,\tilde m_k)\big)(m_{i,k}-\tilde m_{i,k})\\ 
            &\geq  \sum_{(i,k)\in\Ixt^*}m_{i,k}\big(I_{\Delta x}[ \tilde v_{i,k+1}](x_i- \Delta t q_{i,k})+\Delta tL(x_i,q_{i,k})-I_{\Delta x}[\tilde v_{i,k+1}](x_i- \Delta t\tilde q_{i,k})-\Delta tL(x_i,\tilde q_{i,k}) \big)\\ 
        &+\sum_{(i,k)\in\Ixt^*}\tilde m_{i,k}\big(I_{\Delta x}[  v_{i,k+1}](x_i-  \Delta t \tilde q_{i,k})+\Delta tL(x_i,\tilde q_{i,k})-I_{\Delta x}[v_{i,k+1}](x_i- \Delta t q_{i,k})-\Delta tL(x_i,q_{i,k})\big)\\
               &\geq 0,
        \end{aligned}
\end{equation*}
and then
     \begin{equation*}
        \begin{aligned}
            &  \sum_{(i,k)\in\Ixt^*}m_{i,k}\big(I_{\Delta x}[ \tilde v_{i,k+1}](x_i- \Delta t q_{i,k})+\Delta tL(x_i,q_{i,k})-I_{\Delta x}[\tilde v_{i,k+1}](x_i- \Delta t \tilde q_{i,k})-\Delta tL(x_i,\tilde q_{i,k}) \big)\\ 
        &+\sum_{(i,k)\in\Ixt^*}\tilde m_{i,k}\big(I_{\Delta x}[  v_{i,k+1}](x_i-  \Delta t \tilde q_{i,k})+\Delta tL(x_i,\tilde q_{i,k})-I_{\Delta x}[v_{i,k+1}](x_i- \Delta t q_{i,k})-\Delta tL(x_i,q_{i,k})\big)\\
            &= 0.
        \end{aligned}
\end{equation*}
Once again using the non-negativity of $m_{i,k},\tilde m _{i,k}\geq 0$ for any $(i,k)\in \Ixt^*$, together with  \eqref{minproperty}, it follows that, for any $(i,k)\in \Ixt ^*$,
\begin{equation*}
    m_{i,k}\big( I_{\Delta x}[ \tilde v_{i,k+1}](x_i- \Delta t q_{i,k})+\Delta tL(x_i,q_{i,k})-I_{\Delta x}[\tilde v_{i,k+1}](x_i- \Delta t \tilde q_{i,k})-\Delta tL(x_i,\tilde q_{i,k}) \big)= 0 
\end{equation*}
and 
\begin{equation*}
   \tilde  m_{i,k}\big(I_{\Delta x}[  v_{i,k+1}](x_i-  \Delta t  \tilde q_{i,k})+\Delta tL(x_i,\tilde q_{i,k})-I_{\Delta x}[v_{i,k+1}](x_i- \Delta t q_{i,k})-\Delta tL(x_i,q_{i,k})\big) = 0. 
\end{equation*}
For $k=0$, we have $m_{i,0}=\tilde m_{i,0}$ for any $i\in \Ix$. Let $i\in\Ix$ be such that $m_{i,0}=\tilde m_{i,0}>0$. Therefore, we have that
    \begin{equation*}
  I_{\Delta x}[ \tilde v_{i,1}](x_i- \Delta tq_{i,0})+\Delta tL(x_i,q_{i,0})-I_{\Delta x}[\tilde v_{i,1}](x_i-\Delta t \tilde q_{i,0})-\Delta tL(x_i,\tilde q_{i,0})  = 0
\end{equation*}
implies, due to \ref{uniquecontrol}, that $q_{i,0}=\tilde q_{i,0}$ for all such  $i\in \Ix$. By definition of the scheme \eqref{eq:MFGsdiscr2},  this yields  $m_{1}=\tilde m_1$. Iterating over $k \in \It^*$, we conclude that  $m=\tilde m$ , and consequently $v=\tilde v$. 
\end{proof}
From \cref{rem equivalent schemes}, \cref{existenceMFGsd1}, and \cref{uniquenessMFGsd2} the following result holds true:
\begin{cor}\label{existence and uniqueness MFGsd2}
    Let us assume \onlyref{continuity,initialdensity,ass:lagrangian}  and   $(\Delta x,\Delta t)>0$ such that \cref{uniquecontrol,DiscreteMonotonicity,condition dxdt for compact support discrete solution} hold. Then there exists a unique  solution to \eqref{eq:MFGsdiscr2}.  
    \end{cor}
From \cref{rem equivalent schemes},  \cref{existence and uniqueness MFGsd2} and \cref{convergenceschemeMFGsD1}, a convergence result for \eqref{eq:MFGsdiscr2} holds: 
\begin{cor}\label{convergenceschemeMFGsD2}
Let us assume \onlyref{continuity,C2bounded,initialdensity,ass:lagrangian}.  Consider a sequence  $\Delta_n=(\Delta x_n,\Delta t_n)>0$  such that $(\Delta x_n,\Delta t_n)\to 0$ and $\frac{\Delta x_n}{\Delta t_n}\to 0$ as $n\to \infty$, and suppose \onlyref{DiscreteMonotonicity,condition dxdt for compact support discrete solution,uniquecontrol} hold for any $(\Delta x_n, \Delta t_n)$.  Let $(m^n,v^n)$  be the unique solution to \eqref{eq:MFGsdiscr2} for the corresponding grid parameters $(\Delta x_n,\Delta t_n)$. Let $m_{\Delta_n}:=M_{\Delta_n}[m^n]$ and $v_{\Delta_n}:=V_{\Delta_n}[v^n]$.
Then there  exists  a weak solution $(m^*,v^*)$ to  the MFGs system, with $m^*$ absolutely continuous, such that, u.t.s.,  subsequence, $m_{\Delta_n}(t)\to m^*(t)$ in $\P_1(\RR^d)$ for any $t\in [0,T]$ and $v_{\Delta_n}\rightharpoonup^* v^*$ in $L^\infty(Q_T)$. 
\end{cor}

\section{Discrete Learning Value Iteration and Policy Iteration algorithms for \eqref{eq:MFGsdiscr2}}\label{sec:algo}

The  non-linear backward-forward structure of \eqref{eq:MFGsdiscr2} requires an effective procedure for the approximation of its solution.
We consider a Learning Value Iteration algorithm, we study its convergence, and propose an accelerated version based on a Policy Iteration procedure.

\subsection{Discrete Learning Value Iteration algorithm  (DLVI)}\label{sec: DLVI}
Under \ref{uniquecontrol}, we define a discrete Learning Value Iteration algorithm (\textbf{DLVI}) for \eqref{eq:MFGsdiscr2} inspired by the fictitious play procedure introduced in \cite{CardaliaguetHadikhanloo} for a continuous first and second order MFGs problem and in \cite{HadikhanlooSilva} for a discrete problem. 

Given  $\ov{m}^{(1)}\in \S(\GG_{\Delta})$, we iterate over  $n\geq 1$:
\begin{myenum}
\item Compute $v^{(n)}$ 
\begin{equation*}
            \begin{aligned}
               &v^{(n)}_{i,k}=\min_{a\in \ov{B}_{\infty}(0,C_H)}\left\{I_{\Delta x}[v^{(n)}_k](x_{i}-\Delta t a)+\Delta tL(x_i,a)\right\}+\Delta t  f(x_i,\ov m^{(n)}_k) \, \textrm{for any } (i,k)\in \Ixt^*,\\
               &v^{(n)}_{i,N_T}=g(x_i,\ov m^{(n)}_{N_T}) \quad\quad \textrm{for any } i\in \Ix.  
            \end{aligned}
\end{equation*}
\item Define $q^{(n)}$  
\begin{equation*}
    q^{(n)}_{i,k}\in \amin_{a\in \ov{B}_{\infty}(0,C_H)}\left\{I_{\Delta x}[v^{(n)}_{k+1}](x_{i}-\Delta t a)+\Delta tL(x_i,a)\right\}\qquad \textrm{for any }(i,k)\in \Ixt^*.
\end{equation*}
\item  Compute $m^{(n+1)}$ 
\begin{equation*}
            \begin{aligned}
               &m^{(n+1)}_{i,k+1}=\sum_{j\in \Ix}\beta_{i}(x_j-\Delta t q^{(n)}_{j,k})m^{(n+1)}_{j,k}  &&\textrm{for any } (i,k)\in \Ixt^*, \\
               &m^{(n+1)}_{i,0}=\int_{E_{i}} \dd [m^*_{0}](x) &&  \textrm{for any }i\in \Ix.
            \end{aligned}
\end{equation*}
\item Update  $\ov m^{(n+1)}$ 
\begin{equation*}
    \ov{m}^{(n+1)}_{i,k}=\left(1-\frac{1}{n+1}\right)\ov{m}^{(n)}_{i,k}+\frac{1}{n+1}m^{(n+1)}_{i,k} \qquad \textrm{for any }(i,k)\in \Ixt.
\end{equation*}
\end{myenum}
We consider the following stop criterion: given a tolerance $tol>0$,
\textbf{DLVI}  stops at  $n>1$ if  
    \begin{equation}\label{L1stopDLVI}
       E_{1}(m^{(n+1)},m^{(n)}) <tol \qquad  \text{and}\qquad E_{\infty}(v^{(n)},v^{(n-1)}) <tol
\end{equation}
and the solution returned by  the algorithm is  $(m^{(n+1)},v^{(n)})$. The above errors
 are defined as 
\begin{equation*}
     E_{1}(m,\tilde m):=\text{Int}_{ \GG_{\Delta x}}|m_{N_T}-\tilde m_{N_T}| \qquad  \text{and}\qquad    E_{\infty}(v,\tilde v):=\max_{i\in\Ix}|v_{i,0}-\tilde v_{i,0}|,  
\end{equation*}
 for any  $m,\tilde m\in \S(\GG_{\Delta})$ and $v,\tilde v \in \BB(\GG_{\Delta})$, where \text{Int}$_{ \GG_{\Delta x}}$ denotes the approximation of the integral on $ \RR^d$ by using the rectangle Rule on $ \GG_{\Delta x}$.

\subsubsection{Probabilistic notation}\label{sec: probabilistic notation}
To streamline the proof of \textbf{DLVI} convergence, we  slightly modify the notation, in order to apply the convergence result given in \cite{HadikhanlooSilva}. We  introduce a collection of definitions and give a probabilistic interpretation of \eqref{eq:MFGsdiscr2}.
Let $\S(\GG_{\Delta x})$ be the set of probability measure over $\Ix$, defined in  \cref{A_fully_discrete_scheme_for_the_continuity_equation}. 
   For any fixed  $j \in\Ix$ and $a \in \RR^d $, we define the map $p^j[a]:\Ix \to \RR $ by
\begin{equation*}
    p^j[a](i)\coloneqq \beta_i(x_j-\Delta t a) \qquad \text{ for any } i \in\Ix.
\end{equation*} 
    For any fixed  $q \in  \BB( \GG_{\Delta}^*;\RR^d)$, we define the map $P[q]:\Ix\times\Ixt^* \to \RR $ by
\begin{equation*}
    P[q](j,i,k)\coloneqq p^j[q_{j,k}](i) \qquad \text{ for any } j,i \in\Ix, k \in  \It^*.
\end{equation*} 
For any fixed  $i\in \Ix$, we define the map $c_{i}:\RR^d\times  \S(\GG_{\Delta})\to \RR$ by
\begin{equation}\label{runningcost}
    c_{i}(a,m)\coloneqq\Delta tL(x_i,a)+ \Delta tf(x_i,m).
\end{equation}
The scheme \eqref{eq:MFGsdiscr2} written with this new notation becomes: find $(m,v) \in \S(\GG_{\Delta})\times \BB(\GG_{\Delta})$ such that
\begin{equation}\label{eq:MFGsdiscrProb}
            \begin{aligned}
               &v_{i,k}=\sum_{j\in\Ix}P[\Tilde q](i,j,k)\big(v_{j,k+1}+c_{i}(\Tilde q_{i,k},m_k) \big) &&\text{for any   }(i,k)\in \Ixt^*, \\
                 &m_{i,k+1}=\sum_{j\in \Ix}P[\Tilde q](j,i,k)m_{j,k}  &&\text{for any   } (i,k)\in \Ixt ^*,\\
               &v_{i,N_T}=g(x_i, m_{N_T}), \qquad m_{i,0}=\int_{E_{i}} \dd [m^*_{0}](x)&&  \text{for any  }i\in \Ix, \\  
            \end{aligned}
\end{equation}
where  $\Tilde q \in  \BB( \GG_{\Delta}^*;\RR^d)$ is such that
\begin{equation}\label{eq:oprimalcontrol}
       \Tilde q_{i,k}= \amin_{a\in \ov{B}_{\infty}(0,C_H)}\left\{\sum_{j\in\Ix}p^i[a](j)\big(v_{j,k+1}+c_{i}(a,m_k)\big)\right\}\qquad \text{for any   }(i,k)\in \Ixt ^*.
\end{equation}
We observe that \eqref{eq:oprimalcontrol} does not depend on $m_k$ because the cost 
$c$ is decoupled and depends on 
$m_k$ only through 
$f$, which does not depend on 
$a$. \textbf{DLVI}, with this probabilistic notation, computes the sequences of marginal distributions $m^{(n)}\in\S_{\Delta }$ and value functions $v^{(n)}\in \BB( \GG_{\Delta})$ as follows.  Given  $\ov{m}^{(1)}\in  \S(\GG_{\Delta})$, we iterate on  $n\geq 1$:
\begin{myenum}
\item Compute $v^{(n)}$ 
\begin{equation*}
            \begin{aligned}
               &v^{(n)}_{i,k}=\inf_{a\in \ov{B}_{\infty}(0,C_H)}\left\{\sum_{j\in\Ix}p^i[a](j)\big(v^{(n)}_{j,k+1}+c_{i}(a,\ov m^{(n)}_k)\big)\right\} &&\text{for any   }(i,k) \in \Ixt^*, \\
               &v^{(n)}_{i,N_T}=g(x_i, \ov m^{(n)}_{N_T}) &&  \text{for any  }i\in \Ix.  
            \end{aligned}
\end{equation*}

\item  Define $q^{(n)}$  
\begin{equation}\label{optimalcontrolDLVI}
    q^{(n)}_{i,k}\in  \amin_{a\in \ov{B}_{\infty}(0,C_H)}\left\{\sum_{j\in\Ix}p^i[a](j)\big(v^{(n)}_{j,k+1}+c_{i}(a,\ov m^{(n)}_k)\big)\right\}\qquad \text{for any   }(i,k)\in \Ixt ^*.
\end{equation}

\item   Compute $m^{(n+1)}$ 
\begin{equation*}
            \begin{aligned}
               &m^{(n+1)}_{i,k+1}=\sum_{j\in \Ix}P[q^{(n)}](j,i,k)m^{(n+1)}_{j,k}  &&\text{for any   }i\in \Ixt^*, \\
               &m^{(n+1)}_{i,0}=\int_{E_{i}} \dd [m^*_{0}](x) &&  \text{for any  }i\in \Ix.\\  
            \end{aligned}
\end{equation*}
\item  Update $\ov m^{(n+1)}$ 
\begin{equation*}
    \ov{m}^{(n+1)}_{i,k}=\left(1-\frac{1}{n+1}\right)\ov{m}^{(n)}_{i,k}+\frac{1}{n+1}m^{(n+1)}_{i,k} \qquad \text{for any } (i,k)\in \Ixt.
\end{equation*}
\end{myenum}
\subsubsection{Probabilistic interpretation of $(\text{MFGs}_d^2)$}
Given $q \in  \BB( \GG_{\Delta}^*;\RR^d)$ and $m_0^*\in \P_1(\RR^d)$, the pair $(m^*_0,P[q])$ induces a probability distribution over $\S(\GG_{\Delta})$ with marginal distribution given by 
\begin{equation*}\label{eq:definition_M_Q_recursive}
\begin{aligned}
m^{m^*_0}_{i,0}[q]&\coloneqq \int_{E_{i}} \dd [m^*_{0}](x) &&\text{for any } i\in \Ix.\\
m^{m^*_0}_{i,k+1}[q]&\coloneqq \sum_{j \in \Ix}P[q](j,i,k)m^{m^*_0}_{j,k}[q]  && \text{for any } (i,k)\in \Ixt^*.
\end{aligned}
\end{equation*}
Given $m\in \S(\GG_{\Delta})$, we define   $J_m :  \BB( \GG_{\Delta}^*;\RR^d)\to \RR$ by 
\begin{align*}
J_m(q)&\coloneqq \sum_{k=0}^{N_T-1} \sum_{j,i \in \Ix} m^{m^*_0}_{j,k}[q]P[q](j,i, k) c_{j} ( q_{j,k} , m_k) + \sum_{j\in \Ix} m^{m^*_0}_{j,N_T}[q]  g(j , m_{N_T}),  \\ 
&=\sum_{k=0}^{N_T-1} \sum_{j \in \Ix}m^{m^*_0}_{j,k}[q] c_{j} ( q_{j,k} , m_k) + \sum_{j\in \Ix} m^{m^*_0}_{j,N_T}[q]   g(j ,m_{N_T}).
\end{align*}
We consider the following  problem: find $\Tilde q  \in  \BB( \GG_{\Delta}^*;\RR^d)$ such that 
\begin{equation}\label{eq:MFGsdiscrJ}
   \Tilde  q \in \amin_{ \substack{q \in   \BB( \GG_{\Delta}^*;\RR^d)\\ ||q||_{L^\infty(\GG_{\Delta}^*)} \leq C_H}}J_m(q) \qquad \text{ with } \quad m=m^{m^*_0}[\Tilde  q].
\end{equation}
\begin{rem}
    Note that if $(v,m)$ is the solution to \eqref{eq:MFGsdiscrProb}, then $\Tilde q$ satisfying \eqref{eq:oprimalcontrol} solves \eqref{eq:MFGsdiscrJ}.
\end{rem}
\begin{rem}
    If $q^{(n)}$ solves \eqref{optimalcontrolDLVI}, then 
    \begin{equation*}
          q^{(n)} \in \amin_{ \substack{q \in   \BB( \GG_{\Delta}^*;\RR^d)\\ ||q||_{L^\infty(\GG^*_{\Delta})} \leq C_H}}J_{\ov m^{(n)}}(q).
    \end{equation*}
\end{rem}

\subsubsection{Convergence Analysis of DLVI}\label{Convergence Analysis of DLVI}
We show that the sequence $(m^{(n)},v^{(n)})$ generated  by \textbf{DLVI} converges to  the solution to \eqref{eq:MFGsdiscr2} (or, equivalently, \eqref{eq:MFGsdiscrProb}).
   To this end, we  recall  an abstract result established in \cite[Theorem 3.1]{HadikhanlooSilva}, that will allow us to prove the convergence of \textbf{DLVI}. 
Let $\mathcal{X}$ and $\mathcal{Y}$  be two Polish spaces and $\mathcal{Z} \subseteq \mathcal{X}$ be a compact set.
Let  $\mathcal{P}(\mathcal{Z})$  the set of probability measures on $\mathcal{Z}$. The set $\mathcal{P}(\mathcal{Z})$ is compact. Let $\mathcal{F}:\mathcal{Z} \times \mathcal{P}(\mathcal{Z}) \to \RR$ be a given continuous function. Given $x_1\in \mathcal{Z}$, set $\ov{\eta}_1:= \delta_{x_1}$ and for any $n\geq 1$ define
\begin{equation}\label{GenFP}
x_{n+1} \in \amin_{x \in \mathcal{Z}} \mathcal{F}(x,\ov \eta_n), \qquad \ov \eta_{n+1} = \frac{1}{n+1} \sum_{k=1}^{n+1} \delta_{x_{k}} = \frac{n}{n+1} \ov \eta_n + \frac{1}{n+1}\delta_{x_{n+1}}.
\end{equation}
We consider the convergence problem of  the sequence $(\ov \eta_{n})$  to some $\tilde{\eta}\in \mathcal{P}(\mathcal{Z})$ satisfying that
\begin{equation}\label{eq:convergentProblem}
\mbox{supp}(\tilde{\eta}) \subseteq \amin_{x\in \mathcal{Z}}\mathcal{F}(x,\tilde{\eta}).
\end{equation}
\begin{theo}\label{Generalisedfictitiousplay}\cite[Theorem 3.1]{HadikhanlooSilva}
Consider the sequence $(x_n,\ov{\eta}_n)$ defined by \eqref{GenFP} and assume:
\begin{itemize}
\item[{\rm(i)}]$\mathcal{F}$ is monotone, i.e. 
 \begin{equation}\label{monotonocityF}
    \int_{\mathcal{Z}}\Big(\mathcal{F}(x,\eta)-\mathcal{F}(x,\eta')\Big) \dd [\eta-\eta'](
x)\geq 0    \qquad \text{for any } \eta,\eta'\in \mathcal{P} (\mathcal{Z}).
\end{equation}
\item[{\rm(ii)}] $\mathcal{F}$ is Lipschitz, when  $ \mathcal{Z}$ is endowed with the distance $d$  and $\mathcal{P} (\mathcal{Z})$ is endowed with the distance $d_1$, and there exists $C>0$ such that 
\begin{equation*}
| \mathcal{F}(x , \eta) - \mathcal{F}(x , \eta') - \mathcal{F}(x' , \eta) + \mathcal{F}(x' , \eta') | \leq C   d(x,x')   d_1 (\eta,\eta'),
\end{equation*}
for any $x$, $x' \in \mathcal{Z}$, and  $\eta$, $\eta' \in \mathcal{P} (\mathcal{Z})$.
\item[{\rm(iii)}] d$(x_{n+1},x_{n})\to 0$ as $n\to \infty$.
\end{itemize}
Then, every limit point $\Tilde{\eta}$ of $\ov{\eta}_n$(there exists at least one) solves \eqref{eq:convergentProblem}.
\end{theo}
In order to apply \cref{Generalisedfictitiousplay}, 
we define $\mathcal{Z}\coloneqq \{q :  \GG_{\Delta}^*\to \RR^d, ||q||_{L^\infty(\GG_{\Delta}^*)} \leq C_H\}$. Given $\eta \in  \mathcal{P} ( \mathcal{Z})$, we define $M^\eta\in   \S(\GG_{\Delta})$ and $\mathcal{F}:\mathcal{Z}\times \mathcal{P} ( \mathcal{Z}) \to \RR $ by
\begin{equation}\label{GenFormofCostFunction}
    M^{\eta}_k \coloneqq \int_{\mathcal{Z}} m^{m^*_0}_k[q]  \dd [\eta](q),  \,\,\, \text{for any }  k\in \mathcal{I}_{\Delta t},\quad \text{and}  \quad  \mathcal{F}(q,\eta)  \coloneqq  J_{M^\eta}(q).
\end{equation}
We define the fictitious play procedure: given $q^{(1)}\in \mathcal{Z}$ and $\eta \in  \mathcal{P} ( \mathcal{Z})$ such that $\ov{\eta}^{(1)}=\delta_{q^{(1)}}$, for $n\geq 1$, define 
\begin{equation}\label{ficitiousplayF}
q^{(n+1)} \in \amin_{q\in  \mathcal{Z}} \mathcal{F}(q,\ov \eta_n), \qquad \ov \eta_{n+1} = \frac{1}{n+1} \sum_{k=1}^{n+1} \delta_{q^{(k)}} = \frac{n}{n+1} \ov \eta_n + \frac{1}{n+1}\delta_{q^{(n+1)}}.
\end{equation}
\begin{rem}
    Notice that $M^{\ov{\eta}_n}=\ov{m}^{(n)}$.
\end{rem}

\subsubsection{Verification of the assumptions of \cref{Generalisedfictitiousplay}}
Let us consider the following assumption: 
\begin{Hassum}[resume]
    \item  \label{lipschitz in m}
    
    There exist  $C_{F,3},C_{G,3}\in L^1_{loc}(\RR^d)$  that are  positive almost everywhere and 
\begin{equation*}
\begin{aligned}
    |F(x,\eta)-F(x,\tilde \eta)|\leq C_{F,3}(x) \dd_1(\eta,\tilde \eta),  \qquad   |G(x,\eta)-G(x,\tilde \eta)|\leq C_{G,3}(x) \dd_1(\eta,\tilde \eta)
    \end{aligned}
\end{equation*}
for a.e. $x\in \RR^d$ and for any $\eta,\tilde \eta \in \P_1(\RR^d)$.
\end{Hassum}

Similarly to the proof of \cite[Lemma 3.2, Lemma 3.3, Lemma 3.4]{HadikhanlooSilva}, we can verify that, under \onlyref{continuity,lipschitz in m,C2bounded,ass:lagrangian,DiscreteMonotonicity,uniquecontrol,initialdensity,condition dxdt for compact support discrete solution}, the problem \eqref{ficitiousplayF} satisfies the assumptions of \cref{Generalisedfictitiousplay}. By applying \cref{Generalisedfictitiousplay} to \eqref{ficitiousplayF}, the following convergence result for \textbf{DLVI} holds true. The proof employs the same technique as in the proof of \cite[Theorem 3.2]{HadikhanlooSilva}.
\begin{theo}\label{thm:convercenceDLVI}
    Let us assume \onlyref{continuity,C2bounded,lipschitz in m,ass:lagrangian,DiscreteMonotonicity,uniquecontrol,initialdensity,condition dxdt for compact support discrete solution}. Let $(q^{(n)},v^{(n)},m^{(n)}, \ov m^{(n)})$ be the sequences defined by {\rm \textbf{DLVI}}. Then $(q^{(n)},v^{(n)},m^{(n)}, \ov m^{(n)})\to ( q,v,m, m)$, where $(m,v)$ is the unique solution  to \eqref{eq:MFGsdiscr2} and $q$ the unique element in $\Lambda_{\Delta}[v]$, for any initial guess $\ov m^{(1)}\in \S(\GG_{\Delta})$.
\end{theo}

\subsection{Discrete Policy Iteration algorithm (DPI)}\label{DPI}
We consider a discrete Policy Iteration algorithm \textbf{(DPI)} to approximate \eqref{eq:MFGsdiscr2}. The algorithm is inspired by policy iteration procedure introduced in \cite{CacaceCamilliGoffi} for a continuous second order MFGs problem. In \textbf{DPI} the HJB equation is replaced by a linearized transport equation that we  discretize using  a classic SL scheme (see \cite{FalconeFerretti}).
 Let $\rho\in \mathcal{C}_{c}^{\infty}(\RR^d)$ be such that $\rho \geq 0 $ and $\int_{\RR^d} \rho(x)\text{d}x=1$. For $\epsilon>0$, we consider the mollifier $\rho_{\epsilon}(x)=\frac{1}{\epsilon^{d}}\rho(\frac{x}{\epsilon})$.
 
Given  ${q}^{(1)}\in \BB( \GG_{\Delta}^*;\RR^d)$, we iterate over  $n\geq 0$:
\begin{myenum}
    \item  Compute $m^{(n)}$
\begin{equation*} 
            \begin{aligned}
               m^{(n)}_{i,k+1}&=\sum_{j\in \Ix}\beta_{i}(x_j-\Delta t{q}^{(n)}_{j,k})m^{(n)}_{j,k}  &&\text{for any  }(i,k)\in \Ixt^*, \\
               m^{(n)}_{i,0}&=\int_{E_{i}}\text{d}[m^*_{0}](x) &&  \text{for any  }i\in \Ix.  
            \end{aligned}
\end{equation*}
\item Compute $v^{(n)}$ 
\begin{equation*}
            \begin{aligned}
               v^{(n)}_{i,k}&= I[v^{(n)}_{k+1}](x_{i}-q^{(n)}_{i,k}\Delta t)+\Delta tL(x_i,q^{(n)}_{i,k})  +\Delta t  f(x_i,m^{(n)}_k) &&\text{for any   }(i,k)\in \Ixt^*, \\
               v^{(n)}_{i,N_T}&=g(x_i,m^{(n)}_{N_T}) &&  \text{for any  }i\in \Ix.   
            \end{aligned}
\end{equation*}
\item Define $D_x v^{(n)}_{\Delta,\epsilon}(\cdot,t):=D_x \rho_{\epsilon}* V_{\Delta}[v^{(n)}](\cdot,t)$, for  $t\in[0,T]$ and set 

\begin{equation*}
     q^{(n+1)}_{i,k}= D_pH(x_i,D_x v^{(n)}_{\Delta,\epsilon}(x_i,t_k)) \qquad \text{for  any } (i,k)\in \Ixt^*.
\end{equation*}
\end{myenum}
We consider the following stop criterion: given a tolerance $tol>0$,
\textbf{DLVI}  stops at  $n>1$ if  
    \begin{equation}\label{L1stopDPI}
       E_{1}(m^{(n)},m^{(n-1)}) <tol \qquad  \text{and}\qquad E_{\infty}(v^{(n)},v^{(n-1)}) <tol,
\end{equation}
and the solution returned by  the algorithm is  $(m^{(n)},v^{(n)})$.
\begin{rem}\label{rem: advantage of DPI}
In the third step of \textbf{DPI}, in order to have a well defined gradient, the convolution with the mollifier is introduced as in  \cite{CarliniSilva2014}. 
The advantage of using the gradient directly for the policy update, rather than solving a minimization problem, is that it eliminates the minimization step, thereby speeding up the algorithm.
\end{rem}

\subsection{Accelerated Discrete Learning Value Iteration Algorithm (ADLVI)}\label{ADLVI} 
We present an accelerated iterative algorithm \textbf{ADLVI} for \eqref{eq:MFGsdiscr2}, which is constructed by coupling \textbf{DLVI} and \textbf{DPI}, and switching from a coarse grid to a fine grid. More specifically, the algorithm consists in implementing \textbf{DLVI} on a fine grid using as initial guess the distribution computed by the \textbf{DPI} implemented on a coarse grid. For this accelerated algorithm we need to have 
$(\Delta x_c,\Delta t _c,tol_c,\epsilon_c)>0$ for the first part of \textbf{ADLVI}, and $(\Delta x_f,\Delta t_f,tol_f)>0$ such that $\Delta x_f\leq \Delta x_c$ and $\Delta t_f\leq \Delta t_c$ for the last one. We set $\Delta_c:=(\Delta x_c,\Delta t _c)$ and $\Delta_f:=(\Delta x_f,\Delta t _f)$.
\textbf{ADLVI} is made up of these following three steps: 
\begin{myenum}
    \item Given an initial guess ${q}^{(1)}\in \mathcal{B}( \GG_{\Delta_c};\RR^d)$ and $\epsilon_c$, we iterate for  $n\geq 0$ using \textbf{DPI}  on the coarse mesh $\mathcal{G}_{\Delta_c}$, until the stop citerion in \eqref{L1stopDPI}, with tolerance $tol_c$, is satisfied at iteration $n_c$, returning the approximate solutions $(m^{(n_c)},v^{(n_c)},q^{(n_c+1)})$. 
   \item Define $\ov m^{(1)}\in \S(\GG_{\Delta_f})$ as  
   \begin{equation}\label{step 2 ADVI}
   \ov{m}^{(1)}=I[m^{(n_c)}] 
   \end{equation}
   where $I:\S(\GG_{\Delta_c})\to \S(\GG_{\Delta_f})$
is a  interpolation operator.
\item Given the initial guess $\ov 
m ^{(1)}\in \S(\GG_{\Delta_f})$ defined in \eqref{step 2 ADVI}, we iterate for   $n\geq 1$ using \textbf{DLVI} on the fine mesh $\GG_{\Delta_f}$, until the stop criterion in \eqref{L1stopDLVI},  with tolerance $tol_f$, is satisfied  at iteration $n_f$. 
\end{myenum}
The approximated solution given by \textbf{ADLVI} is $(m^{(n_f+1)},v^{(n_f)})$.

The goal of developing \textbf{ADLVI} is to exploit the low computational cost of \textbf{DPI} (see \Cref{rem: advantage of DPI}) to construct an initial guess 
$\ov m ^{(1)}$  for the convergent algorithm \textbf{DLVI} (see \cref{thm:convercenceDLVI}), so that it requires fewer iterations than it would with a random initialization. Numerically, we will observe that, for a fixed grid refinement, \textbf{DLVI} requires more iterations to converge than \textbf{ADLVI}, making \textbf{ADLVI} computationally more efficient.

\section{Numerical tests}\label{sec:tests}

We show the performance of the proposed algorithms
by approximating the solution to three different MFGs systems. In the first problem, the exact solution is known and it allows us to study the accuracy of \textbf{DLVI}. In this case, we focus on the one-dimensional setting.  
In the second and third problem, which are one and  two-dimensional respectively, we compare the performance of \textbf{DLVI} and \textbf{ADLVI}.
We denote by $(m^L,v^L)$ and $(m^{ACC},v^{ACC})$ the solutions computed by the two methods.
We initialize the iterations of \textbf{DLVI} with $\overline m^{(1)}_{i,k}=\int_{E_i}\dd [m_0^*](x)$ for any $(i,k)\in \Ixt$, and the iterations of \textbf{ADLVI} with $q^{(1)}_{i,k}=\partial_pH(x_i,\nabla G(x_i,m_0^*))$ for any $(i,k)\in \mathcal{I}^*_{\Delta _c}$.
In all the proposed tests, we choose   $tol=1E-6$. Finally, the numerical simulations have been carried
out in MATLAB R2022a on a computer with processor AMD Ryzen 5 5600H with Radeon Graphics 3.30 GHz.
We approximate the cell average of $F$, $G$, and $m_0$ over  $E_i$ by their pointwise values at the cell center $x_i$. 
Thanks to the regularity of the data, this approximation is consistent with first-order accuracy.

 \subsection{Example with analytical solution}\label{test1}
 We consider the MFGs system with quadratic Hamiltonian $H(x,p)=|p|^2/2$, coupling terms 
\begin{equation*}
    F(x,\eta)=\frac{1}{2}\left(x-\int_{\RR^d}y\dd[\eta](y) \right)^{2}, \quad G(x,\eta)=0,
\end{equation*}
and  initial measure $m^*_0$ given as  the density of a Gaussian random vector with mean $\mu^*_0\in\RR^d$ and covariance matrix $\Sigma_0\in\ \RR^{d\times d}$, assumed, for simplicity,
to be diagonal. 
This problem does not satisfy {\rm(H2)} and {\rm(H3)}, however, it admits a unique explicit solution $(m^*,v^*)$ (see \cite{CalzolaCarliniSilva2024}) that allows us to study the numerical convergence rates of \textbf{DLVI}.

In the following simulations, we consider this example set in the domain $[-2,2]^d\times[0,T]$. We choose $d=1$, $T=1$, $\mu_0=0.1$ and $\Sigma _{0}=0.5$.

We test the performance of  \textbf{DLVI} for $\Delta t=\sqrt{\Delta x}/2$ and for $\Delta t=\Delta x^{2/3}/2$, in combination with several refinements of the space grid.  The time step $\Delta t=\sqrt{\Delta x}/2$ is 
chosen to minimize the truncation error term $\mathcal{O}\left(\Delta t +\frac{\Delta x}{\Delta t}\right)$ which appears in the proof of Theorem \ref{convergenceschemeMFGsD1}.
Figure \ref{fig: convergence rate sqrt} displays the error behavior between the exact and numerical solution, $E_1(m^L, \hat m^*)$ and $E_\infty(v^L, \hat v^*)$, under successive space grid step refinements. The plot in the
left shows that the error $E_1(m^L, \hat{m}^*)$ follows a convergence rate of order $0.523$, while the right plot shows that $E_\infty(v^L, \hat{v}^*)$ has a convergence rate $0.482$, both represented by the slope of the triangles. Therefore, the overall convergence rate is approximately $1/2$, which is consistent with the truncation error.
The errors $E_1(m^L, \hat m^*)$ and $E_\infty(v^L, \hat v^*)$ corresponding to the  the choice $\Delta t = \Delta x^{2/3}/2$ proposed in \cite{CarliniSilvaZorkot} for the same problem, are shown in Figure \ref{fig: convergence rate}.
The left plot   shows that the error $E_1(m^L, \hat{m}^*)$ follows a convergence rate of order  $0.623$, while the right plot $E_\infty(v^L, \hat{v}^*)$  follows a convergence rate of order $0.579$. 
These rates overestimate the  predicted order given by this choice, which corresponds to $\mathcal{O}\left(\Delta x^{1/3}\right)$, 
  suggesting that the truncation error obtained in  Theorem \ref{convergenceschemeMFGsD1}  might  not be sharp and that an improved rate could, in principle, be obtained.

 \begin{figure}[!t]
        \centering        
        \includegraphics[width=4.7cm]{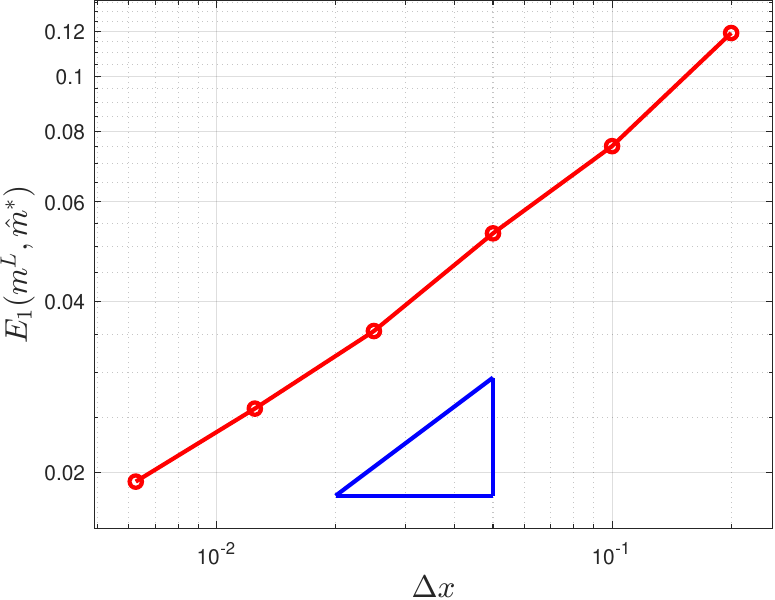}
                \hspace{5mm}
        \includegraphics[width=4.7cm]{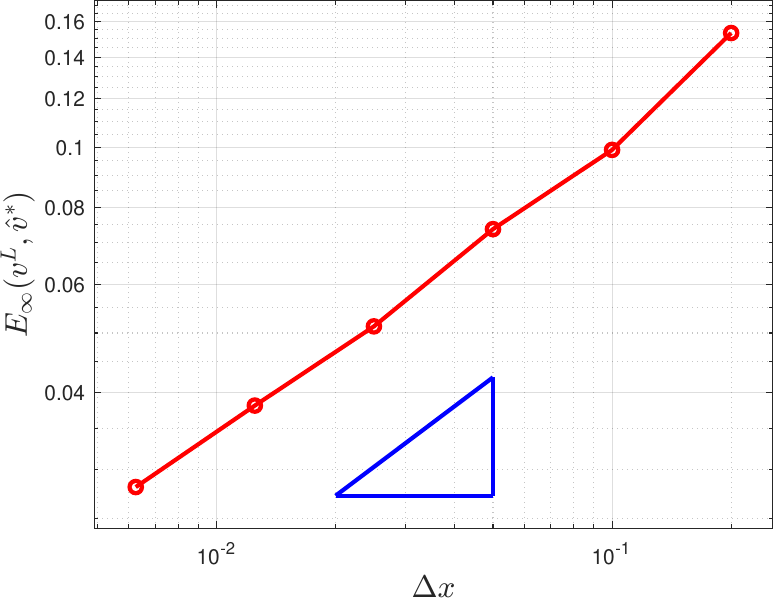}
     \caption{Example  \ref{test1} with $\Delta t=\sqrt{\Delta x}/2$. Convergence plots of $E_1(m^L,\hat m^*)$ (left) and of  $E_\infty(v^L,\hat v^*)$ (right)  in logarithm scale.}\label{fig: convergence rate sqrt}
    \end{figure}

 \begin{figure}[!t]
        \centering        
        \includegraphics[width=4.7cm]{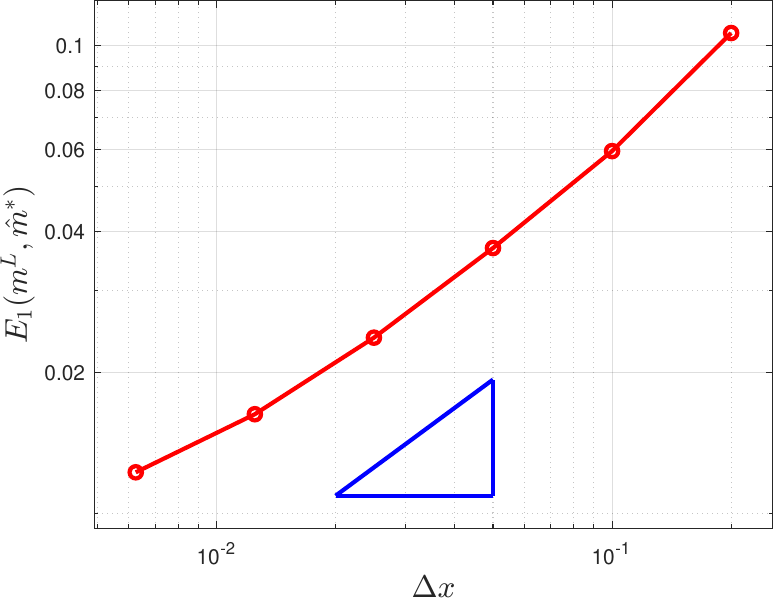}
        \hspace{5mm}
        \includegraphics[width=4.7cm]
        {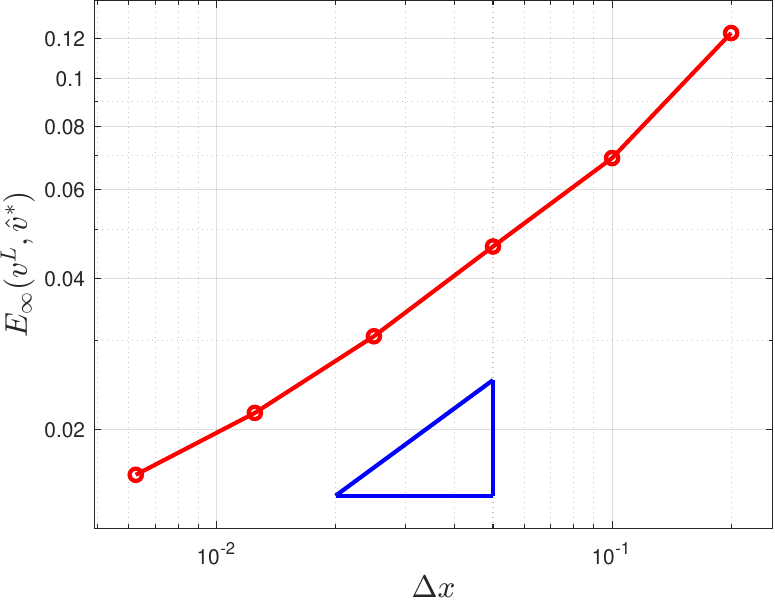}
     \caption{Example  \ref{test1} with $\Delta t=\Delta x^{2/3}/2$. Convergence plots  of $E_1(m^L,\hat m^*)$ (left) and of  $E_\infty(v^L,\hat v^*)$ (right) in logarithm scale.}\label{fig: convergence rate}
    \end{figure}

 \begin{figure}[!t]
        \centering        
        \includegraphics[width=4.7cm]{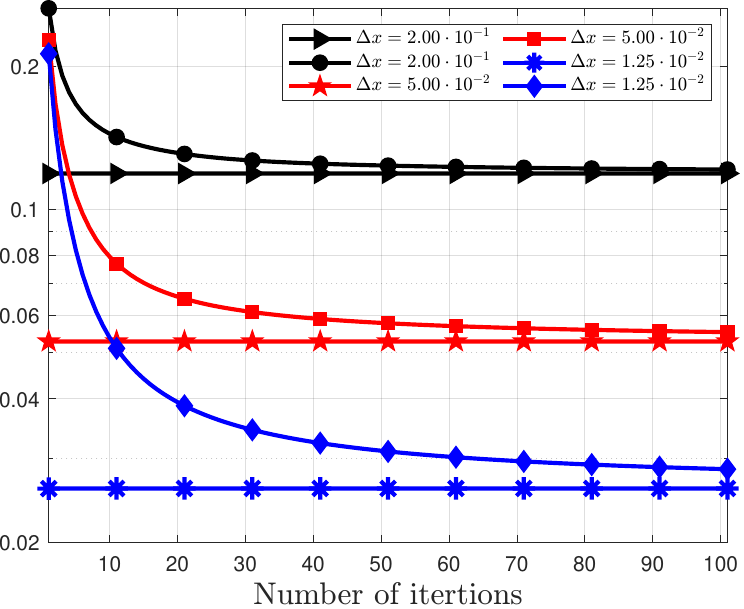}
        \hspace{5mm}
        \includegraphics[width=4.7cm]
        {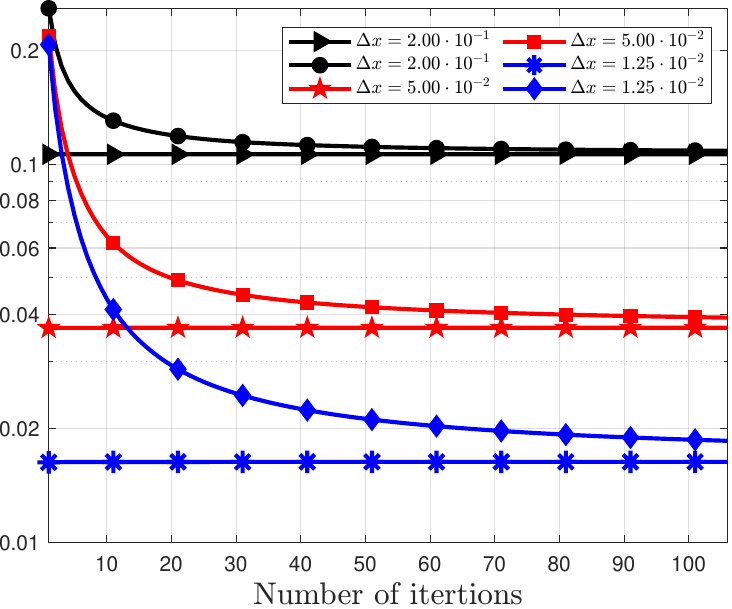}
     \caption{Example  \ref{test1} for $\Delta t=\sqrt{\Delta x}$ (left) and for $\Delta t=\Delta x^{2/3}/2$ (right). History of the errors between the current iterate and the exact solution. }\label{fig: convergence
history of error between current iterate and exact solution}
    \end{figure}
We analyze the convergence with respect the iterations   of the learning algorithm \textbf{DLVI} by computing  the difference between the current iterate $n$ and the exact solution, testing the algorithm for several spatial grid refinements while keeping the total number of iteration $N$ fixed.
Specifically, we set $N=101$ for $\Delta t=\sqrt{\Delta x}/2$ and $N=106$ for $\Delta t=\Delta x^{2/3}/2$. The corresponding results are presented in Figure \ref{fig: convergence history of error between current iterate and exact solution}, on the left and right, respectively. 
In Figure \ref{fig: convergence history of error between current iterate and exact solution}, the curves marked with triangles, stars, and asterisks illustrate the behavior of $E_1(m^{(n+1)}, \hat{m}^*)$ over the iterations $n \in \{1, \dots, N\}$ for $\Delta x = 2.00 \cdot 10^{-1}$, $\Delta x = 5.00 \cdot 10^{-2}$, and $\Delta x = 1.25 \cdot 10^{-2}$, respectively. Similarly, the curves marked with circles, squares, and diamonds represent the behavior of $E_1(\overline{m}^{(n+1)}, \hat{m}^* )$ for the same iterations and space step.
We observe that $E_1(m^{(n+1)}, \hat{m}^*)$ remains nearly constant, indicating that $m^{(n)}$ is already close to the discrete solution of the scheme \eqref{eq:MFGsdiscr1}   after a single  iteration. In contrast, $E_1(\overline{m}^{(n+1)}, \hat{m}^*)$ exhibits a decreasing trend, converging toward the plateau defined by $E_1(m^{(N)}, \hat{m}^*)$.  This behavior confirms that, due to the smoothing procedure, $\overline{m}^{(n)}$ also progressively converges toward to the discrete solution of the scheme \eqref{eq:MFGsdiscr1}, but at a slower rate. Moreover, Figure \ref{fig: convergence history of error between current iterate and exact solution} shows  that these plateau values decrease as $\Delta x$ decreases, 
 in agreement with the decreasing error trends observed in the left plots of Figure \ref{fig: convergence rate sqrt} and of Figure \ref{fig: convergence rate}.

Next, we proceed to compare \textbf{DLVI} and \textbf{ADLVI}. 
We first point out that, given the smoothness of the solution, in the third step of \textbf{DPI} within the \textbf{ADLVI} algorithm, the gradient is approximated using central finite differences, making convolution-based regularization unnecessary.
We implement  \textbf{ADLVI} with  $\Delta x_c=\Delta x_f=\Delta x$,  $\Delta t_c=\Delta t_f=\Delta t$, and 
$tol_c=tol_f=tol$. We compare \textbf{DLVI} and \textbf{ADLVI} across several refinements of the space grid by evaluating the errors $E_1(m^L,m^{ACC})$ and $E_\infty(v^L,v^{ACC})$, along with the computational time and the number of iterations. Table \ref{table: test with analytical solution sqrt} and  Table \ref{table: test with analytical solution} correspond to the time steps $\Delta t=\sqrt{\Delta x}/2$ and $\Delta t=\Delta x^{2/3}/2$, respectively. In both tables, we display the space grid step $\Delta x$, the errors $E_1(m^L,m^{ACC})$ and $E_\infty(v^L,v^{ACC})$, the number of iterations $n_L$ for \textbf{DLVI}, the pair $(n_c,n_f)$ for \textbf{ADLVI}, and the corresponding computational times. 
In both cases, \textbf{ADLVI} demonstrates a clear numerical advantage. Although both algorithms produce comparable solutions, the number of iterations $n_f$ required by \textbf{ADLVI} is smaller than the number of iterations $n_L$ needed by \textbf{DLVI}.
Furthermore, since the iterations $n_c$ involve a substantially lower computational cost, the overall runtime of \textbf{ADLVI} is reduced compared to that of \textbf{DLVI}.

\begin{table}[!t]
  \caption{Example  \ref{test1}  with $\Delta t=\sqrt{\Delta x}/2$. Space grid steps, errors, number iterations, and cpu times.}
        \renewcommand\arraystretch{1.1}
        \centering
        \begin{tabular}{lllllll}
                    \toprule
               $\Delta x$      & $ E_1(m^L,m^{ACC})$  & $E_\infty(v^L,v^{ACC})$  & $n_L$      &$(n_c,n_f)$ & time \textbf{DLVI} & time \textbf{ADLVI}\\ \midrule
        $2.00 \cdot 10^{-1}$ & $4.60 \cdot 10^{-5}$  &  $1.08 \cdot 10^{-4}$ & $101$   & $(20,19)$ & $4.83s$ & $0.927s$  
                \\ 
                $1.00 \cdot 10^{-1}$&    $4.63 \cdot 10^{-5}$  &   $1.01 \cdot 10^{-4}$  & $76$   & $(15,4)$ & $10.1s$ & $0.577s$ 
                \\ 
                  $5.00 \cdot 10^{-2}$  &  $3.85 \cdot 10^{-5}$  &  $8.20 \cdot 10^{-5}$  &  $56$ & $(20,8)$ & $19.4s$ & $2.91s$  
                \\ 
                $2.50 \cdot 10^{-2}$  &   $3.12 \cdot 10^{-5}$ &  $6.28 \cdot 10^{-5}$ &   $43$  & $(17,8)$ & $51.1s$ & $15.5s$  
                \\ 
                $1.25 \cdot 10^{-1}$  &  $2.66 \cdot 10^{-5}$  & $5.01 \cdot 10^{-5}$  &  $34$ & $(17,7)$ & $108s$ & $25.3s$  
                \\ 
                $6.25\cdot 10^{-3}$  & $2.26 \cdot 10^{-5}$ &  $4.23 \cdot 10^{-5}$ &$28$  & $(22,6)$ & $295s$ & $71.9s$        \\  \bottomrule
        \end{tabular}

        \label{table: test with analytical solution sqrt}
    \end{table}

\begin{table}[!t]
  \caption{Example  \ref{test1}  with $\Delta t=\Delta x^{2/3}/2$. Space grid steps, errors, number iterations, and cpu times.}
        \renewcommand\arraystretch{1.1}
        \centering
        \begin{tabular}{lllllll}
                    \toprule
               $\Delta x$      & $ E_1(m^L,m^{ACC})$   & $E_\infty(v^L,v^{ACC})$ & $n_L$      &$(n_c,n_f)$ & time \textbf{DLVI} & time \textbf{ADLVI}\\ \midrule
  $2.00 \cdot 10^{-1}$   & $5.65 \cdot 10^{-5}$ &$1.24 \cdot 10^{-4}$ &  $106$ & $(12,11)$ & $6.24s$ & $0.691s$  
                \\ 
                $1.00 \cdot 10^{-1}$  &  $5.28 \cdot 10^{-5}$ &  $1.14 \cdot 10^{-4}$  &  $79$ & $(9,9)$ & $15.1s$ & $1.86s$  
                \\ 
                  $5.00 \cdot 10^{-2}$ & $4.15 \cdot 10^{-5}$  &$8.48 \cdot 10^{-5}$  &  $58$ & $(8,8)$ & $34.0s$ & $4.82s$  
                \\ 
                $2.50 \cdot 10^{-2}$ & $3.25 \cdot 10^{-5}$  & $6.30 \cdot 10^{-5}$ & $44$  & $(8,6)$ & $91.1s$ & $12.6s$  
                \\ $1.25 \cdot 10^{-1}$  &  $2.86 \cdot 10^{-5}$ & $4.88 \cdot 10^{-5}$ &  $35$  & $(8,5)$ & $246s$ & $35.1s$  
                \\ 
                $6.25\cdot 10^{-3}$ &$2.21 \cdot 10^{-5}$ &  $4.06 \cdot 10^{-5}$ &  $29$ & $(8,4)$ & $755s$ & $104s$       \\ \bottomrule
        \end{tabular}

        \label{table: test with analytical solution}
    \end{table}

\subsection{Examples with a target position and aversion to crowded regions}\label{test2}
We consider two MFGs systems in dimensions $d=1$ and $d=2$ with state-independent and state-dependent Hamiltonians, respectively,  with coupling terms:  
\begin{equation*} 
F(x,\eta)= \lambda\min\{|x-\ov{x}|^{2},R\} +(r_{\sigma}*\eta)(x),\quad G(x,\eta)=0,
\end{equation*}
where $\lambda,
R,\sigma >0$, $\ov{x}\in\RR^d$, and $r_\sigma(x)=\frac{1}{(2\pi \sigma^2)^{d/2}}e^{-\frac{|x|}{2\sigma^2}}$.
We observe that these functions satisfy {\rm(H1)}, {\rm(H2)}, {\rm(H5)},  {\rm(H6)} and {\rm(H9)}. 
These tests are inspired by  similar examples in \cite{Lauriere21, CarliniSilvaZorkot}.
In this game, agents aim to reach the target point $\overline{x}$, driven by the first component of the cost function $F$, while simultaneously being adverse to the presence of other players, as penalized by the second component of $F$.
The following simulations show the influence of the positive parameter $\lambda$, which weights the objective of reaching the target point.
We consider as initial density
\begin{equation*} 
m^*_{0}(x)= \frac{1}{2\int_{\RR^d}m^*(x)\dd x} \left (m^*(x-1)+ m^*(x+1)\right),
\end{equation*} 
where 
\begin{equation*}
    m^*(x)=(1-4|x|^2)^3\chi_{\ov B(0,1/2)}(x).
\end{equation*}
 This initial density satisfies {\rm(H3)}, and describes a configuration with agents  distributed in two regions of radius $1/2$, as illustrated in Figure  \ref{fig:initial dencity}
 (left plot  for the case $d=1$,  central and  right plots for the case 
 $d=2$).
Moreover, we consider $\ov x=0$,  $\sigma=0.5$ and $R=9$. The problems are posed on the spatial domain $[-3,3]^d$, while the approximated densities are displayed over the restricted space domain $[-2,2]^d$ to focus on the region where  the density is effectively supported.
 \begin{figure}[!h]
        \centering        
        \includegraphics[width=4cm]{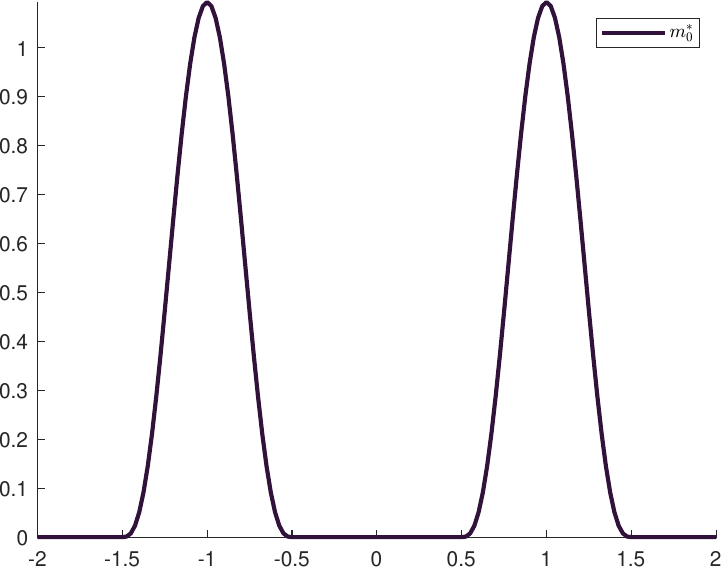}
        \includegraphics[width=4.7cm]{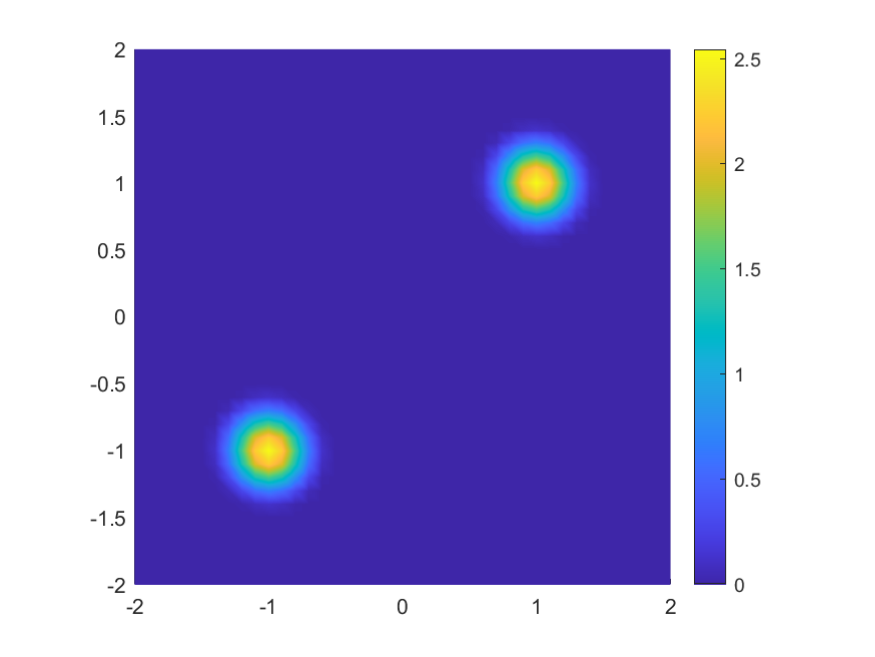}
        \includegraphics[width=4.7cm]{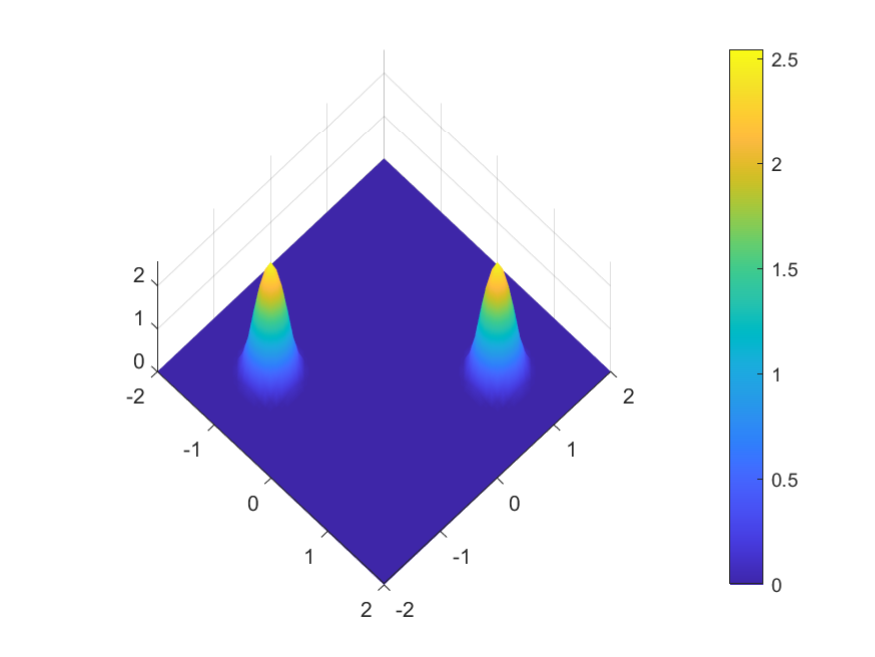} \caption{Example  \ref{test2}. Initial density for $d=1$ (left) and for $d=2$ (center and right). }\label{fig:initial dencity}
    \end{figure}    
\subsubsection{
Case $d=1$ with state-independent quadratic Hamiltonian
}\label{test2_1d}
In this example, we consider the above problem with quadratic Hamiltonian $H(x,p)=|p|^2/2$ and dimension $d=1$. 
In the following simulations, we set $T=4$, $\Delta x=2.5\cdot10^{-2}$ and $\Delta t=\Delta x^{2/3}/2$. We apply \textbf{DLVI} to solve the problem with $\lambda=2.5$ and $\lambda =0.8$, and we show the time evolution  of the density $m^L_{10\cdot k}$ for $k=0,\cdots,\lfloor N_T/10 \rfloor$ in Figure  \ref{fig: evolution 1d lambda25}, left and  right plots, respectively. In the first case,  the mass concentrates in the target point, since for $\lambda=2.5$ the term penalizing the distance to the target dominates over the aversion to crowded regions, as illustrated in  Figure  \ref{fig: evolution 1d lambda25} (left). In contrast,  Figure  \ref{fig: evolution 1d lambda25} (right) shows that  the mass does not concentrate in the target and the two parts of the density remain separate with lower peaks, since for $\lambda=0.8$ the aversion term has a dominant influence on the dynamics.  Using the same test settings, we apply \textbf{ADLVI} with $\Delta x_c=2\Delta x$,  $\Delta t_c=2\Delta t$,  $\epsilon_c=3\sqrt{\Delta t_c}$, 
$tol_c=tol_f=tol$, $\Delta x_f=\Delta x$,   $\Delta t_f=\Delta t$,  and we choose $\rho(x)=\frac{1}{\sqrt{2\pi}}e^{-|x|^2/2}$ as mollifier in the convolution. We compare  \textbf{ADLVI}
with  \textbf{DLVI}
by computing the errors
$E_1(m^L,m^{ACC})$ and $E_\infty(v^L,v^{ACC})$, as well as the computational time and number of iterations.
In  Table  \ref{table: test1d},  we report $E_1(m^L,m^{ACC})$, $E_\infty(v^L,v^{ACC})$, the number of iterations $n_L$ for \textbf{DLVI}, the pair $(n_c,n_f)$ for \textbf{ADLVI}, and the corresponding computational times. The first and second rows of the table correspond to the cases 
$\lambda=2.5$ and $\lambda=0.8$, respectively. In both cases, \textbf{ADLVI} is advantageous from a numerical point of view. In fact, although both algorithms yield comparable solutions,  $n_f$ is smaller than the number of iterations $n_L$ required by \textbf{DLVI} and,  thanks to the low cost of the $n_c$ iterations, the total computational time of \textbf{ADLVI} is lower than that of \textbf{DLVI}.

Finally, it is worth noting that the factor $1/(n+1)$ in step (iv) of \textbf{DLVI}, while ensuring the convergence of the algorithm, makes the
convergence slow in practice, as already observed in  example \ref{test1} Figure \ref{fig: convergence
history of error between current iterate and exact solution}. As a numerical attempt to accelerate the convergence, we introduce an accelerated variant, denoted as \textbf{ADLVI+}. The structure of \textbf{ADLVI+} is identical to that of \textbf{ADLVI}, except for the final step, where \textbf{DLVI} is executed on the fine mesh and the average density is updated according to
$\overline{m}^{(n+1)}_{i,k}=2/3\overline{m}^{(n)}_{i,k}+1/3{m}^{(n+1)}_{i,k}$. 
The weights $2/3$ and $1/3$ are chosen empirically through experimentation.
We denote by $(m^{ACC+}, v^{ACC+})$ the solution computed by \textbf{ADLVI+} and by $(n_c^+,n_f^+)$ the number of iterations performed by it.
We compare \textbf{DLVI} with \textbf{ADLVI+} by evaluating the errors $E_1(m^L, m^{ACC+})$ and $E_\infty(v^L, v^{ACC+})$, as well as the computational time and the number of iterations. The results are reported in Table \ref{table: test1d +}. As in the previous cases, \textbf{ADLVI+} achieves faster convergence, further reducing the CPU time.   Moreover, \textbf{ADLVI+}  reduces on average the CPU time by $97.55\%$, whereas the \textbf{ADLVI} reduces the CPU time by $75 .55\%$.
 \begin{figure}[!t]
        \centering 
            \includegraphics[width=4.5cm]{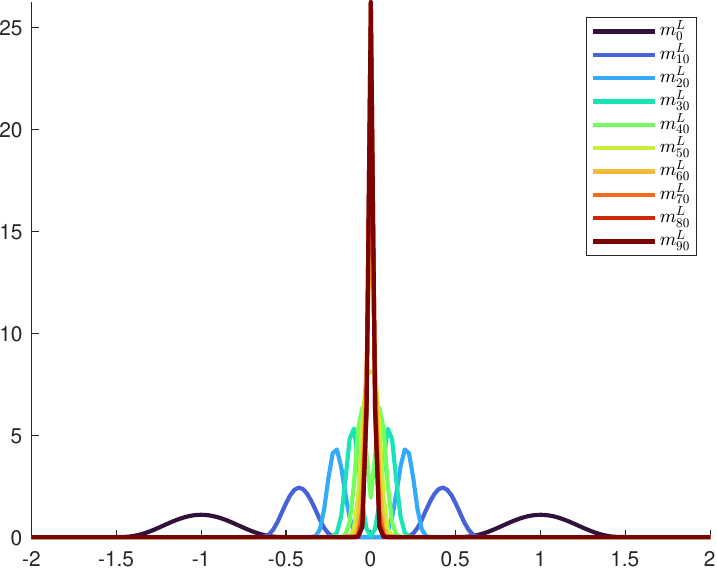}       
        \hspace{1cm}
            \includegraphics[width=4.5cm]{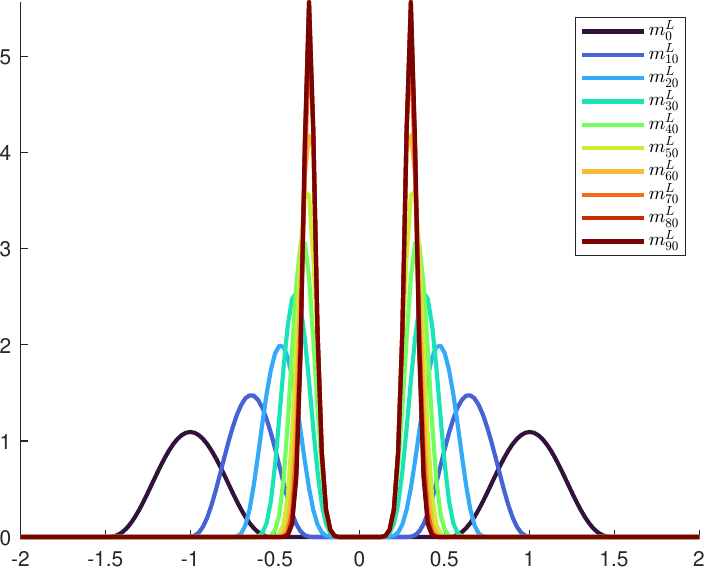}
        \caption{Example  \ref{test2_1d}. Time-evolution of the density for $\lambda=2.5$ (left) and for $\lambda=0.8$ (right).} \label{fig: evolution 1d lambda25}
\end{figure}

 \begin{table}[!t]
  \caption{Example  \ref{test2_1d}. Parameter $\lambda$, errors, number iterations, and cpu times.}
        \renewcommand\arraystretch{1.1}
        \centering
        \begin{tabular}{lllllll}
                    \toprule
               $\lambda $     & $ E_1(m^L,m^{ACC})$  & $E_\infty(v^L,v^{ACC})$   & $n_L$      &$(n_c,n_f)$ & time \textbf{DLVI} & time \textbf{ADLVI}\\ \midrule
              $2.5$  &  $1.70 \cdot 10^{-3}$   & $1.50 \cdot 10^{-3}$ & $1410$ & $(16,148)$ & $9.54 \cdot 10^{3}s$ &$1.00 \cdot 10^{3}s$      \\ 
              $0.8$ & $6.68 \cdot 10^{-4}$  &$5.41 \cdot 10^{-4}$ & $875$ & $(16,320)$ &  $4.97 \cdot 10^{3}s$  & $1.81 \cdot 10^{3}s$ \\ 
                 \bottomrule
        \end{tabular}
       
        \label{table: test1d}
    \end{table}
    
 \begin{table}[!t]
  \caption{Example  \ref{test2_1d}. Parameter $\lambda$, errors, number iterations, and cpu times.}
        \renewcommand\arraystretch{1.1}
        \centering
        \begin{tabular}{lllllll}
                    \toprule
               $\lambda $     & $ E_1(m^L,m^{ACC+})$  & $E_\infty(v^L,v^{ACC+})$   & $n_L$      &$(n_c^+,n_f^+)$ & time \textbf{DLVI} & time \textbf{ADLVI+}\\ \midrule
               $2.5$    &  $1.61 \cdot 10^{-3}$ & $1.38 \cdot 10^{-3}$  & $1410$ & $(16,22)$ & $9.54 \cdot 10^{3}s$ &$1.52\cdot 10^{2}s$      \\ 
              $0.8$  & $1.65 \cdot 10^{-3}$ &$8.41 \cdot 10^{-4}$ & $875$ & $(16,29)$ &  $4.97 \cdot 10^{3}s$  & $1.63 \cdot 10^{2}s$  \\ 
                 \bottomrule
        \end{tabular}
       
        \label{table: test1d +}
    \end{table}

\subsubsection{Case $d=2$ with state-dependent quadratic Hamiltonian}\label{test2_2d}
In this example, we consider the above problem in dimension $d=2$ and with  Hamiltonian $H(x,p)=|p|^2/2+b(x)\cdot p$, where 
\begin{equation*}
    b(x)=\gamma(-x_2,x_1), \qquad \text{for any }x=(x_1,x_2)\in\RR^2
\end{equation*}
is a rotating velocity field and $\gamma>0$. 
We observe that the Hamiltonian $H(x,p)$ does not satisfy all the required assumptions. However, the drift can be regularized by multiplying the linear term by a suitable mollifier, chosen so that the Hamiltonian remains unchanged within the selected spatial domain. 
We set $T=2.5$, $\gamma=2.5$, $\Delta x=1.00\cdot10^{-1}$ and $\Delta t=\Delta x^{2/3}/2$. As in the previous problem, we consider two  tests for two different values of $\lambda$. 
In the first test, we apply \textbf{DLVI}  with $\lambda=2$ and we show the time evolution  of the density $m^L_{ k}$ for $k=0,\cdots, N_T$, using a 2D representation,  on the left of Figure  \ref{fig: test2d_2D}. A 3D representation of the final density $m^L_{N_T}$, together with the contours of its time evolution, is displayed on the left of Figure  \ref{fig: test2d_3D}.
Analogously to the one-dimensional case, we observe that the mass concentrates at the target point, as the attraction toward the target outweighs the effect of crowd aversion for $\lambda=2$. However, in this setting, the two groups of agents approach the target while rotating clockwise, due to the influence of the rotational velocity field.  In the second test,
we apply  \textbf{DLVI}  with $\lambda =0.8$. We display on the right of Figure \ref{fig: test2d_2D} the time evolution of the density $m^L_{ k}$ for $k=0,\cdots, N_T$, using a 2D visualization. A 3D representation of the final density $m^L_{N_T}$, together with the contours of its time evolution, is displayed on the right of Figure  \ref{fig: test2d_3D}.
In this case, the agents still move toward the target while rotating clockwise. However, they do not reach the target and remain split into two separate groups, as the influence of crowd aversion prevails over the attraction to the target. 

We now evaluate the performance of \textbf{ADLVI} under the same conditions used in the previous tests  Using the same test settings, we apply \textbf{ADLVI} with  $\Delta x_c=\Delta x_f=\Delta x$,  $\Delta t_c=\Delta t_f=\Delta t$, $tol_f=tol_c=tol$,  $\epsilon_c=4\sqrt{\Delta t_c}$ and we choose $\rho(x)=\frac{1}{2\pi}e^{-|x|^2/2}$ as mollifier in the convolution. We compare  \textbf{ADLVI} with \textbf{DLVI}.
The results are summarized in  Table  \ref{table: test2d}, where we report the values of $E_1(m^L,m^{ACC})$, $E_\infty(v^L,v^{ACC})$, the number of iterations $n_L$ for \textbf{DLVI}, the pair $(n_c,n_f)$ for \textbf{ADLVI}, and the corresponding computational times for both methods. The first row refers to the case  $\lambda=2$ and the second to the case  $\lambda=0.8$. In both situations, \textbf{ADLVI} achieves better numerical efficiency. Although the two algorithms produce solutions of comparable accuracy,   $n_f$ is lower  than the number of iterations $n_L$ of \textbf{DLVI} and,  since  the $n_c$ iterations have a very low cost, the total computational time of \textbf{ADLVI} remains below that of \textbf{DLVI}.
\begin{figure}[!b]
        \centering     
        {
            \includegraphics[width=5.2cm]{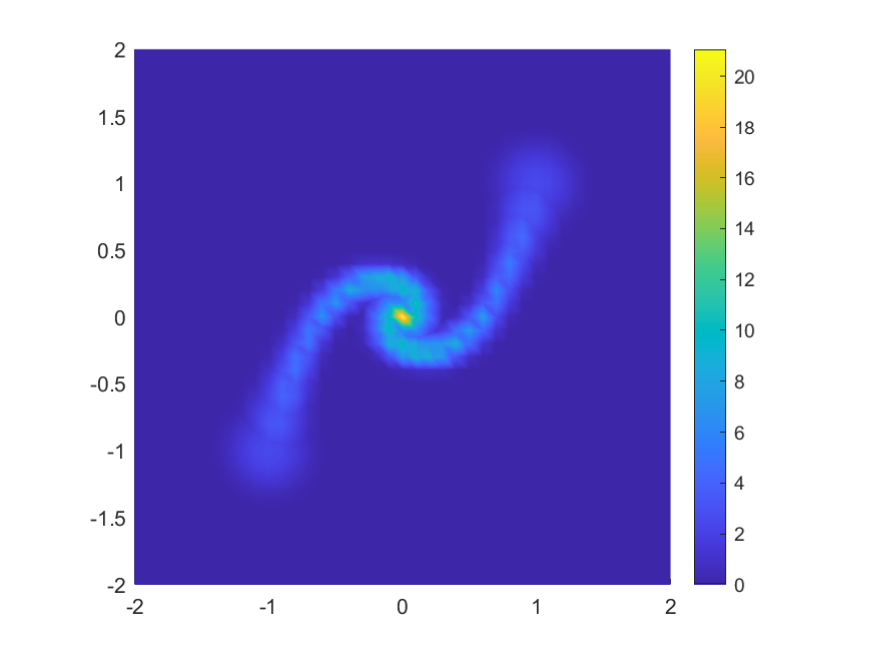}
        }       
        {
            \includegraphics[width=5.2cm]{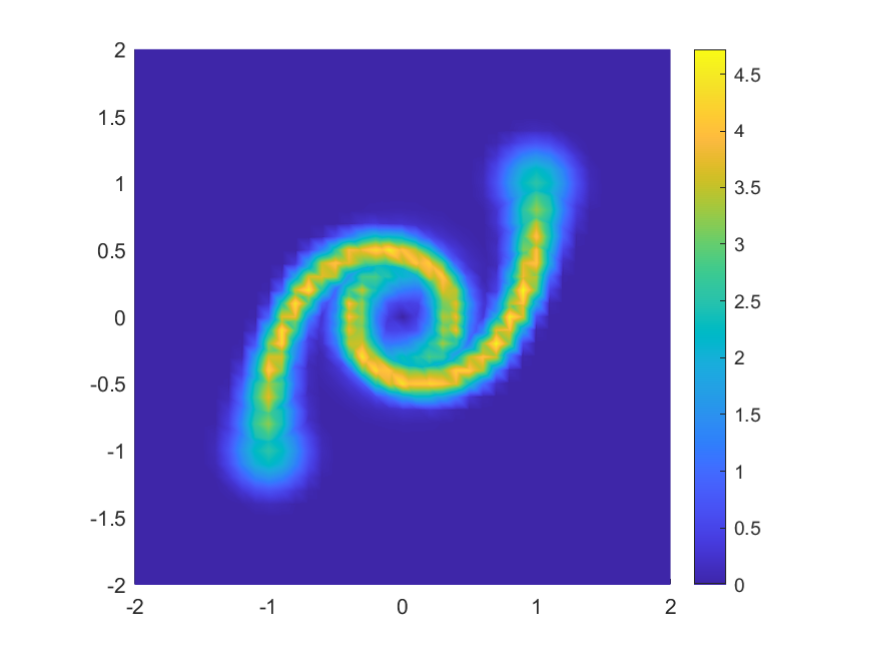
            }
        }
        \caption{Example  \ref{test2_2d}.  Time-evolution (2D view) of the density, $\lambda=2$ (left) and   $\lambda=0.8$ (right).  \label{fig: test2d_2D}}
    \end{figure}

\begin{figure}[!t]
       \centering
        {
            \includegraphics[width=7cm]{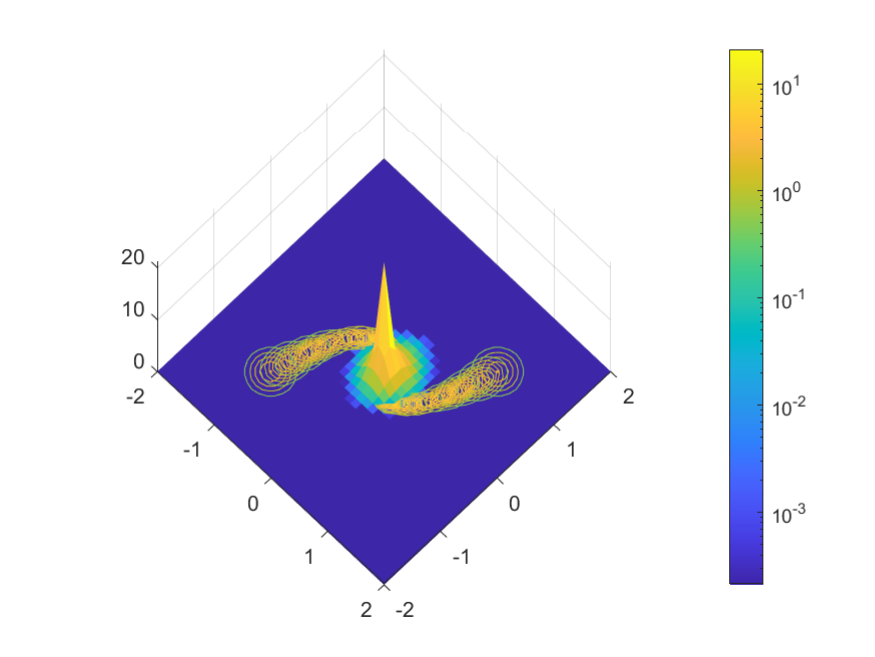}
        } 
        {
            \includegraphics[width=7cm]{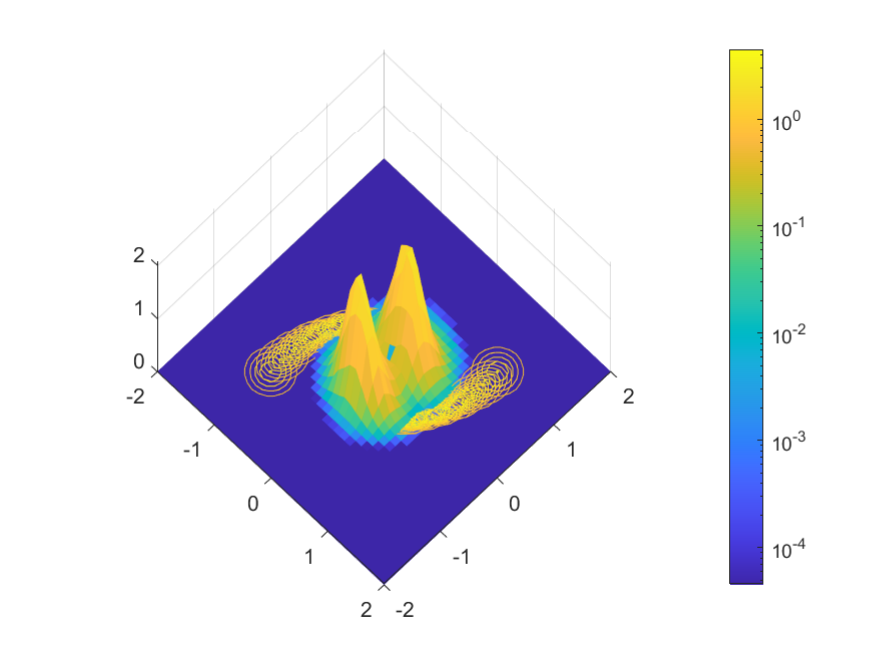}
        }
        \caption{ Example  \ref{test2_2d}.  Final  density (3D view), $\lambda=2$ (left) and   $\lambda=0.8$ (right).}\label{fig: test2d_3D}
    \end{figure}

   \begin{table}[!t]
 \caption{Example  \ref{test2_2d}. Parameter $\lambda$, errors, number iterations, and cpu times.}
        \renewcommand\arraystretch{1.1}
        \centering
        \begin{tabular}{lllllll}
                    \toprule
               $\lambda $     & $ E_1(m^L,m^{ACC})$  & $E_\infty(v^L,v^{ACC})$   & $n_L$      &$(n_c,n_f)$ & time \textbf{DLVI} & time \textbf{ADLVI}\\ \midrule
               $2$  & $7.06 \cdot 10^{-4}$   &    $1.26 \cdot 10^{-3}$ &$909$      &  $(16,435)$  &  $7.45\cdot 10^{3}s$         & $3.73\cdot 10^{3}s$\\ \hline
                 $0.8$ &  $1.34 \cdot 10^{-6}$&       $7.42 \cdot 10^{-4}$ & $712$  & $(13,386)$  & $5.89\cdot 10^{3}s$           & $3.35\cdot 10^{3}s$  \\ \bottomrule
        \end{tabular}
       
        \label{table: test2d}
    \end{table}

\appendix 

\section{Appendix}\label{Appendix}
\begin{proof}[Proof of \cref{Adelta properties}] 
Let $w=(m,v),\tilde w=(\tilde m,\tilde v),q,\tilde q, \mu ,\tilde \mu$ be as in the statement. Since $m$ and $\tilde m$ have compact support, there exists $\tilde C>0$ such that if $x_i\notin \ov B(0,\tilde C)$ then $m_{i,k}=\tilde m_{i,k}=0$ for any $k\in\It$. As a consequence, the sum over $i\in \Ix$ appearing in $ \langle  A_{\Delta }[\tilde w]-A_{\Delta }[ w], \tilde w-w \rangle_{\Delta }$, is effectively restricted to the finite set $ 
\{i\in\Ix | x_i\in  B(0,\tilde C+\sqrt{d}TC_H)\neq \emptyset \}$.
Combining this observation with \eqref{optimalrelaxedcontrolproperty}, we can write $ \langle  A_{\Delta }[\tilde w]-A_{\Delta }[ w], \tilde w-w \rangle_{\Delta }$ as follows 
\begin{equation*}
    \begin{aligned}
          \langle  A_{\Delta }[\tilde w]-A_{\Delta }[ w], \tilde w-w \rangle_{\Delta }&= \sum_{i\in \Ix} \big((\tilde v_{i,0}-v_{i,0})(m_{i,0}-\tilde m_{i,0}) +(v_{i,N_T}-\tilde v_{i,N_T})(m_{i,N_T}-\tilde m_{i,N_T}) \big)\\ 
         &+\Delta t\sum_{k\in\It^*}\sum_{i\in \Ix}\big(f(x_i,m_k)-f(x_i,\tilde m_k)\big)(m_{i,k}-\tilde m_{i,k})\\
        &+ \sum_{k\in \It^*} \sum_{i\in\Ix} \Big[
		   \int_{\RR^d}\Big(I_{\Delta x}[v_{k+1}](x_i - \Delta t a) + \Delta tL(x_i,a)\Big)\dd [\mu_{i,k}](a)\Big. \Big. \\
		& -\int_{\RR^d}\Big(I_{\Delta x}[\tilde{v}_{k+1}](x_i - \Delta t a) +  \Delta tL(x_i,a)\Big)\dd [\tilde \mu_{i,k}](a) \Big]  (m_{i,k} - \tilde{m}_{i,k})
		\\
		&+  \sum_{k\in \It^*} \sum_{i\in\Ix}\Big[  - \sum_{j\in\Ix}  m_{j,k}\int_{\RR^d}\beta_i(x_j-\Delta t a)\dd [\mu_{j,k}](a) \\ &  + \sum_{j\in\Ix} \tilde{m}_{j,k}\int_{\RR^d} \beta_i(x_j-\Delta t a)\dd[\tilde \mu_{j,k}](a)  \Big] (v_{i,k+1} - \tilde{v}_{i,k+1} ). \\
    \end{aligned}
\end{equation*}
Finally, \eqref{eq:equality Adelta} is proved by recalling  the definition of the interpolation operator \eqref{interpolation operator} and by observing that we can commute the summations, since they are finite. 
In addition,
using  \eqref{minproperty} and the fact that $m_{i,k},\tilde m _{i,k}\geq 0$ for any $(i,k)\in \Ixt$,  we obtain that 
 \begin{equation*}
        \begin{aligned}
           & \langle  A_\Delta[\tilde w]-A_\Delta[ w], \tilde w-w \rangle_{\Delta} \\
            &\geq  \sum_{i\in \Ix}\big((-v_{i,0}+\tilde v_{i,0})(m_{i,0}-\tilde m_{i,0}) +(v_{i,N_T}-\tilde v_{i,N_T})(m_{i,N_T}-\tilde m_{i,N_T}) \big)+\\
            &+\Delta t\sum_{k\in\It^*}\sum_{i\in \Ix}\big(f(x_i,m_k)-f(x_i,\tilde m_k)\big)(m_{i,k}-\tilde m_{i,k}).
        \end{aligned}
    \end{equation*}
If \cref{DiscreteMonotonicity} holds,  this  inequality implies \eqref{eq:  inequality Adelta discrmon}.

\end{proof}
\begin{lem}\label{Compact Lemma 2.4 CapuzzoDolcetta}
 Let  $\Omega\subset\RR^d$ be a compact space, $w\in\CC(\Omega)$ and suppose that $x^0$ is the strict minimum point of $w$ in $\Omega$. If  $w^n\in \CC(\Omega)$ converges  uniformly to $w$ in $\Omega$, then for any  $\{x^n\}$ such that
    \begin{equation*}
        w^n(x_n)\leq w^n(x) \quad \forall x \in   \Omega
        \end{equation*}
        (there exists at least one), the following holds true 
    \begin{equation*}
       \lim_{n\to \infty} x^{n}= x^0.
        \end{equation*}
\end{lem}
 \begin{lem}\label{Compact nostrict Lemma 2.4 CapuzzoDolcetta} 
      Let $\Omega\subset\RR^d$ be a compact space and $w\in \CC(\Omega)$. If $w_n\in \CC(\Omega)$ converges  uniformly to $w$ in $\Omega$, then for any sequence  $\{x_n\}_n$ such that
    \begin{equation*}
        w_n(x_n)\leq w_n(x) \quad \forall x \in   \Omega
        \end{equation*}
    (at least one exists) and, for any convergent subsequences $\{x_{n_k}\}_k$ (at least one exists), there exists a point $\tilde x \in \Omega$ such that
    \begin{equation*}
       \lim_{k\to \infty} x_{n_k}=\tilde x, \quad and  \quad w(\tilde x)\leq w(x) \quad \text{for any } x \in \Omega.
        \end{equation*}
\end{lem}
\cref{Compact Lemma 2.4 CapuzzoDolcetta} and \cref{Compact nostrict Lemma 2.4 CapuzzoDolcetta} are proved as \cite[Lemma 2.4]{BardiDolcetta}.

\bibliographystyle{abbrv}   
\bibliography{referencerevised}

\end{document}